\documentclass[a4paper, 11pt] {article}
\usepackage{amsfonts, amsmath, amssymb, amsthm,mathtools, url,dsfont}
\usepackage{graphicx}
\usepackage[english] {babel}
\usepackage{enumitem}
\usepackage{array}
\usepackage{lineno}
\usepackage{tikz}
\usepackage{float}
\usepackage{cellspace}
\newtheorem{thm}{Theorem}[section]
\newtheorem*{thm*}{Theorem}
\newtheorem{cor}[thm]{Corollary}

\newtheorem{lem}[thm]{Lemma}

\newtheorem{prop}[thm]{Proposition}

\theoremstyle{remark}
\newtheorem{rem}[thm]{Remark}

\newtheorem{cond}[thm]{Condition}

\DeclareMathOperator{\spa}{span}

\DeclareMathOperator{\diam}{diam}

\newcommand{\R}{\mathbb{R}}

\newcommand{\Z}{\mathbb{Z}}
\newcommand{\La}{\mathbb{L}}

\newcommand{\Q}{\mathbb{Q}}
\newcommand{\eps}{\varepsilon}
\DeclareMathOperator{\conv}{conv}

\DeclareMathOperator{\indre}{int}

\DeclareMathOperator{\tr}{Tr}
\newcommand{\Ha}{\mathcal{H}}

\newcommand{\II}{\textit{II}}

\DeclareRobustCommand\altoverline[2][0.8]{%
\mathmakebox[\widthof{#2}][c]{\overline{\mathmakebox[\widthof{#2}*\real{#1}][c]{#2}}}}

\title{Local digital algorithms applied to Boolean models}

\author{Julia H\"{o}rrmann, Anne Marie Svane\\ \\ Dedicated to Eva B. Vedel Jensen on the occasion\\ of her 65th birthday }

\date{}
\begin{document}
\maketitle
\begin{abstract}
We investigate the estimation of specific intrinsic volumes of stationary Boolean models by local digital algorithms; that is, by weighted sums of $n\times\ldots \times n$ configuration counts. We show that asymptotically unbiased estimators for the specific surface area or integrated mean curvature do not exist if the dimension is at least two or three, respectively. For $3$-dimensional stationary, isotropic Boolean models, we derive asymptotically unbiased estimators for the specific surface area and integrated mean curvature. 
For a Boolean model with balls as grains we even obtain an asymptotically unbiased estimator for the specific Euler characteristic. This solves an open problem from \cite{nagel}.
\end{abstract}


\section{Introduction}

Let $Z\subseteq \R^d$ be a geometric object. We model a black-and-white digital image of $Z$ as the set $Z\cap \La$ where $\La$ is some observation lattice. The set $Z\cap \La$ can be thought of as the set of foreground or black pixels (voxels), while $\La \backslash Z$ corresponds to the background or white pixels (voxels). This is illustrated in Figure \ref{dig1}.

Given this information, we want to derive geometric information about $Z$. Of particular interest are the intrinsic volumes of $Z$, including such natural quantities as volume, surface area, integrated mean curvature, and Euler characteristic.  A variety of algorithms for their estimation has been suggested in the literature, see e.g.\ \cite{spodarev,digital,lindblad,OM,nagel}. Many of these algorithms are of local type, depending only on the local configurations of black and white points occurring in the image. Such local algorithms are often chosen in practice because they are  intuitive and simple to write down explicitly. Moreover, their computation time is only linear in the number of voxels, see \cite{OM} for more information on implementations.

Local algorithms have been studied theoretically in the design based setting where $Z$ is a deterministic set and $\La$ is stationary random. With the exception of volume and Euler characteristic, results show that they are almost always biased \cite{kampf2,am3,johanna}, even asymptotically when the resolution goes to infinity. 

In this paper we study local algorithms when applied to a random set, more precisely a Boolean model.
Boolean models are the basic models from stochastic geometry for the description of porous structures, e.g.\ in physics, material science, or biology. There exist several monographs that treat Boolean models; \cite{SW} contains the mathematical theory from stochastic geometry, \cite{chiu} also presents many applications, and \cite{molchanov} puts an emphasizes on available statistical methods. 
  
We compare the mean estimates of the specific intrinsic volumes of a Boolean model to the true value. Results from stochastic geometry allow us a more explicit quantification of the bias than in the deterministic case. The idea was already outlined in \cite{nagel} when $Z$ is a stationary isotropic Boolean model, and the authors use it to compute the asymptotic bias of a specific 3D algorithm as the grid width goes to zero. In \cite{am}, the approach of \cite{nagel} is used in 2D not only to compare known algorithms but also to derive general formulas for the bias in high resolution and to give an optimal algorithm. We are going to generalise this approach to 3D.

We start by considering a stationary, but not necessarily isotropic, Boolean model and derive formulas for the mean digital estimators  up to second order in the grid width.
The foundation for this is an asymptotic formula as the grid width $a$ goes to zero for the hit-and-miss probabilities
\begin{equation*}
P(aB\subseteq Z,aW \subseteq \R^d \backslash Z),
\end{equation*}
where  $Z$ is a stationary Boolean model in $\R^d$ with compact convex grains and $B,W\subseteq \R^d$ are finite sets. 

The resulting formulas for the mean digital estimators generalise the formulas of \cite{nagel} and \cite{am} to non-isotropic grain distributions. They have a resemblance to the Miles formulas \cite{miles} for specific intrinsic volumes, but contain a rotation bias. 
The first order asymptotics are similar to the corresponding result in the design based setting with the difference that the deterministic set is replaced by the Blaschke body associated with $Z$. In contrast to this, a new term shows up in the second order formulas due to the underlying randomness. The formulas lead to the first main result. 
\begin{thm*}
Let $Z\subseteq\R^d$ be a stationary Boolean model satisfying Condition~\ref{cond2}. Then, there exists no asymptotically unbiased estimator for the specific surface area or integrated mean curvature based on $n\times\ldots\times n$-configuration counts if the dimension is at least two or three, respectively.
\end{thm*}
Next, we concentrate on Boolean models that are also isotropic. Specializing to three dimensions and $2\times 2 \times 2$-configuration counts, we obtain the following result. 
\begin{thm*}
Let $Z\subseteq\R^3$ be a stationary, isotropic Boolean model satisfying Condition~\ref{cond2}. Then, there exist asymptotically unbiased estimators for the specific surface area and integrated mean curvature based on $2\times 2\times 2$-configuration counts. Possible weights are given in Table \ref{optimal}.
\end{thm*}
In the case of a 3D Boolean model where the grains are balls of a random radius which is almost surely bounded from below, results of \cite{markus} allow us a more detailed analysis. Thus, we can derive third order formulas for the mean estimators. Thereby, we can describe the asymptotic mean values for the full set of estimators in 3D. In particular, we obtain an asymptotically unbiased estimator for the specific Euler characteristic, which solves an open problem from \cite{nagel}.
\begin{thm*}
Let $Z\subseteq\R^3$ be a Boolean model with balls as grains. Then, there exists an asymptotically unbiased estimator for the specific Euler characteristic based on $2\times 2\times 2$-configuration counts. Possible weights are given in Table~\ref{optimal}.
\end{thm*}

 Applying our results to the algorithms suggested in \cite{nagel}, we can even show that they all have a bias already in the second order terms. Instead, our algorithms based on the weights in Table \ref{optimal} are optimal up to third order.


The paper is structured as follows: We start by collecting some background material in Section \ref{Prelim}. Then we compute asymptotic formulas for the hit-and-miss probabilities in Section \ref{hitmiss}. In Section \ref{dig}, local algorithms are defined formally. Then the results of Section \ref{hitmiss} are used to draw conclusions about the estimators. In Section \ref{isotropy} we specialise to the case of isotropic Boolean models. The special case of Boolean models with balls as grains is studied more deeply in Section \ref{Balls}, leading to an optimal algorithm given in Subsection \ref{algor}. 

\begin{figure}[h]
\begin{equation*}
\begin{tikzpicture}[scale=1.5]
\draw[fill] (1.3,0.3) circle [radius=0.4];
\draw[fill] (-0.2,0.4) circle [radius=0.4];
\draw[fill] (0.5,1.15) circle [radius=0.4];
\draw[fill] (-0.3,1.4) circle [radius=0.4];
\draw[fill] (1.4,1.2) circle [radius=0.4];
\draw[fill] (0.1,1.7) circle [radius=0.4];
\draw[fill] (2,1.65) circle [radius=0.4];
\draw[fill] (0.7,0.5) circle [radius=0.4];
\draw[fill] (2.2,0.4) circle [radius=0.4];

\draw[fill] (-0.5,0) circle [radius=0.025];
\draw[fill] (-0.5,0.5) circle [radius=0.025];
\draw[fill] (-0.5,1) circle [radius=0.025];
\draw[fill] (-0.5,1.5) circle [radius=0.025];
\draw[fill] (-0.5,2) circle [radius=0.025];
\draw[fill] (0,0) circle [radius=0.025];
\draw[fill] (0,0.5) circle [radius=0.025];
\draw[fill] (0,1) circle [radius=0.025];
\draw[fill] (0,1.5) circle [radius=0.025];
\draw[fill] (0,2) circle [radius=0.025];
\draw[fill] (0.5,0) circle [radius=0.025];
\draw[fill] (0.5,0.5) circle [radius=0.025];
\draw[fill] (0.5,1) circle [radius=0.025];
\draw[fill] (0.5,1.5) circle [radius=0.025];
\draw[fill] (0.5,2) circle [radius=0.025];
\draw[fill] (1,0) circle [radius=0.025];
\draw[fill] (1,0.5) circle [radius=0.025];
\draw[fill] (1,1) circle [radius=0.025];
\draw[fill] (1,1.5) circle [radius=0.025];
\draw[fill] (1,2) circle [radius=0.025];
\draw[fill] (1.5,0) circle [radius=0.025];
\draw[fill] (1.5,0.5) circle [radius=0.025];
\draw[fill] (1.5,1) circle [radius=0.025];
\draw[fill] (1.5,1.5) circle [radius=0.025];
\draw[fill] (1.5,2) circle [radius=0.025];
\draw[fill] (2,0) circle [radius=0.025];
\draw[fill] (2,0.5) circle [radius=0.025];
\draw[fill] (2,1) circle [radius=0.025];
\draw[fill] (2,1.5) circle [radius=0.025];
\draw[fill] (2,2) circle [radius=0.025];
\draw[fill] (2.5,0) circle [radius=0.025];
\draw[fill] (2.5,0.5) circle [radius=0.025];
\draw[fill] (2.5,1) circle [radius=0.025];
\draw[fill] (2.5,1.5) circle [radius=0.025];
\draw[fill] (2.5,2) circle [radius=0.025];

\draw[black] (4,0) circle [radius=0.07];
\draw[fill] (4,0.5) circle [radius=0.07];
\draw[black] (4,1) circle [radius=0.07];
\draw[fill] (4,1.5) circle [radius=0.07];
\draw[black] (4,2) circle [radius=0.07];
\draw[black] (4.5,0) circle [radius=0.07];
\draw[fill] (4.5,0.5) circle [radius=0.07];
\draw[black] (4.5,1) circle [radius=0.07];
\draw[fill] (4.5,1.5) circle [radius=0.07];
\draw[fill] (4.5,2) circle [radius=0.07];
\draw[black] (5,0) circle [radius=0.07];
\draw[fill] (5,0.5) circle [radius=0.07];
\draw[fill] (5,1) circle [radius=0.07];
\draw[fill] (5,1.5) circle [radius=0.07];
\draw[black] (5,2) circle [radius=0.07];
\draw[black] (5.5,0) circle [radius=0.07];
\draw[fill] (5.5,0.5) circle [radius=0.07];
\draw[black] (5.5,1) circle [radius=0.07];
\draw[black] (5.5,1.5) circle [radius=0.07];
\draw[black] (5.5,2) circle [radius=0.07];
\draw[fill] (6,0) circle [radius=0.07];
\draw[fill] (6,0.5) circle [radius=0.07];
\draw[fill] (6,1) circle [radius=0.07];
\draw[fill] (6,1.5) circle [radius=0.07];
\draw[black] (6,2) circle [radius=0.07];
\draw[black] (6.5,0) circle [radius=0.07];
\draw[black] (6.5,1) circle [radius=0.07];
\draw[fill] (6.5,0.5) circle [radius=0.07];
\draw[fill] (6.5,1.5) circle [radius=0.07];
\draw[fill] (6.5,2) circle [radius=0.07];
\draw[black] (7,0) circle [radius=0.07];
\draw[fill] (7,0.5) circle [radius=0.07];
\draw[black] (7,1) circle [radius=0.07];
\draw[black] (7,1.5) circle [radius=0.07];
\draw[black] (7,2) circle [radius=0.07];
\end{tikzpicture}
\end{equation*}
\caption{Digital image of a Boolean model with balls as grains. }
\label{dig1}
\end{figure}
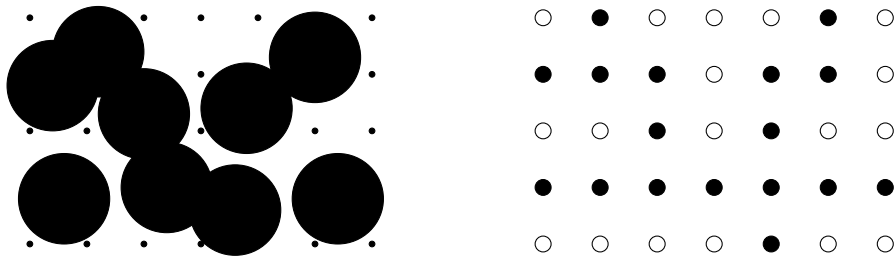

\section{Preliminaries}\label{Prelim}
By a Boolean model $Z$ we shall always mean a stationary Boolean model in $\R^d$ with compact convex grains of common distribution $\Q$ and intensity $\gamma>0$. That is, 
\begin{equation*}
Z = \bigcup_{i=1}^\infty (\xi_i + K_i)
\end{equation*}
where  $\{\xi_1,\xi_2,\dots \}$ is a stationary Poisson process in $\R^d$ with intensity $\gamma$ and $K_1, K_2, \dots$ is a sequence of i.i.d. random compact convex subsets (convex bodies) of $\R^d$ with distribution $\Q$ independent of $\{\xi_1,\xi_2,\dots \}$ and satisfying the integrability condition
\begin{equation}\label{intcond}
\int V_d(K \oplus B^d) \Q(dK) < \infty.
\end{equation}
Here $B^d$ is the unit ball in $\R^d$, $\oplus $ is the Minkowski addition and $V_d$ is the $d$-dimensional volume. 

The intrinsic volumes are important functionals of convex geometry, see \cite{schneider}.
They are the unique functionals $V_q$, $q=0,\ldots,d$ on the space of convex bodies which fulfill the Steiner formula
\begin{equation}\label{steiner}
V_d(K\oplus \epsilon B^d)=\sum\limits_{q=0}^d \epsilon^{d-q}\kappa_{d-q}V_q(K),
\end{equation}
where $K$ is a convex body, $\epsilon>0$ and $\kappa_q:= \frac{\pi^{q/2}}{\Gamma(1+\frac{q}{2})}$ is the volume of $B^q$.
In particular $V_d$ is the volume, $2V_{d-1}$ is the $(d-1)$-dimensional surface area, $2\pi(d-1)^{-1}V_{d-2}$ is the integrated mean curvature and $V_0$ is the Euler characteristic.

In stochastic geometry it has proven useful to consider spatial and probabilistic averages, so called specific intrinsic volumes, see \cite{SW}. 
The specific intrinsic volumes of $Z$ are defined by
\begin{equation}\label{specdef}
\altoverline{V}_{q}(Z)= \lim_{r\to \infty} \frac{EV_q(Z\cap rA)}{V_{d}(rA)}
\end{equation}
where $V_q$, $q=0,\dots,d$ and $A$ is a compact convex window with non-empty interior.
An alternative description is 
\begin{equation}\label{specdef2}
\altoverline{V}_{q}(Z)=EV_q(Z\cap [0,1]^d) - EV_q(Z\cap \partial^+ [0,1]^d)
\end{equation}
where $\partial^+ [0,1]^d= [0,1]^d \backslash [0,1)^d$. The effect of both, the limit in \eqref{specdef} and the subtraction in \eqref{specdef2}, is that the contributions to $V_q$ coming from the boundary of the window are removed. 

We shall mainly be interested in the cases $q=d,d-1,d-2$. 
Let $V_{d-1,d-1}$ denote the mixed functional of translative integral geometry, c.f.\ \cite[Sections 5.2 and 6.4]{SW}, which is defined via the translative integral formula
\begin{align*}
&\int\limits_{\R^d}V_{d-2}(K_1\cap (K_2+x))dx\\
&= V_d(K_1)V_{d-2}(K_2)+V_d(K_2)V_{d-1}(K_1)+V_{d-1,d-1}(K_1,K_2)
\end{align*}
for convex bodies $K_1$, $K_2$.
By $K$ we denote the typical grain, i.e. the random convex body with distribution $\Q$. 
Let $K_1$ and $K_2$ be two independent copies of $K$. 
The specific intrinsic volumes in the cases $q=d,d-1,d-2$ can now be expressed in terms of mean intrinsic volumes of $K$ and the mean mixed functional of $K_1$ and $K_2$, namely it holds by \cite[Theorem 9.1.5]{SW}:
\begin{align}\label{specific}
\altoverline{V}_d(Z){}&= 1- e^{-\gamma EV_d(K)}\\ \nonumber
\altoverline{V}_{d-1}(Z){}&= e^{-\gamma EV_d(K)}\gamma EV_{d-1}(K)\\ \nonumber
\altoverline{V}_{d-2}(Z) {}&= e^{-\gamma EV_d(K)}\bigg(\gamma EV_{d-2}(K) - \frac{\gamma^2}{2}EV_{d-1,d-1}(K_1,K_2)\bigg).
\end{align}

 If the grain distribution is isotropic,  then
\begin{equation*}
EV_{d-1,d-1}(K_1,K_2)= \frac{(d-1)\kappa_{d-1}^2}{d\kappa_d\kappa_{d-2}}EV_{d-1}(K)^2
\end{equation*}
 by the principal kinematic formula \cite[Theorem 2.2]{kiderlen}. 

In the special case of a stationary Boolean model with isotropic grain distribution in 3D, the specific intrinsic volumes are given by the Miles formulas \cite{miles} or \cite[Theorem 9.14]{SW}:
\begin{align}\label{Miles}
\altoverline{V}_3(Z){}&= 1- e^{-\gamma EV_3(K)}\\ \nonumber
\altoverline{V}_{2}(Z){}&= e^{-\gamma EV_3(K)}\gamma EV_{2}(K)\\  \nonumber
\altoverline{V}_{1}(Z) {}&= e^{-\gamma EV_3(K)}\bigg(\gamma EV_{1}(K) - \frac{\gamma^2\pi}{8}EV_{2}(K)^2\bigg)\\ \nonumber
\altoverline{V}_{0}(Z) {}&= e^{-\gamma EV_3(K)}\bigg( \gamma EV_0(K)-\frac{\gamma^2}{2}EV_2(K)EV_1(K) + \frac{\gamma^3\pi}{48}EV_2(K)^3 \bigg).
\end{align}

\section{Hit-and-miss probabilities for stationary Boolean models}\label{hitmiss}
In this section we derive the theoretical results for hit-and-miss probabilities which we will need in later sections for the study of digital algorithms applied to Boolean models.

Let $B,W\subseteq \R^d$ be two finite sets that are not both empty. We consider the hit-and-miss probabilities 
\begin{equation*}
P(aB\subseteq Z, aW \subseteq \R^d\backslash Z)
\end{equation*}
when $a>0$ is small.
By the inclusion-exclusion principle,
\begin{align}\nonumber
P({B\subseteq Z  , W \subseteq \R^d \backslash Z })&= P(W \subseteq \R^d \backslash Z )-P\Big(\bigcup_{b\in B} \{\{b\}\cup W \subseteq \R^d \backslash Z \}\Big)\\ \nonumber
&=P( W \subseteq \R^d \backslash Z ) + \sum_{\emptyset \neq S\subseteq B} (-1)^{|S|}P(S\cup W \subseteq \R^d \backslash Z)\\ \label{EI1}
&= \sum_{S\subseteq B} (-1)^{|S|}P(S\cup W \subseteq \R^d \backslash Z).
\end{align}

For a compact set $C\subseteq \R^d$ it is well known, see e.g. \cite[(9.3) and (9.4)]{SW}, that
\begin{equation}\label{exp}
P(aC \subseteq \R^d \backslash Z ) = e^{-\gamma E V_d({K}\oplus a\check{C})}
\end{equation}
where $\check{C} = \{-c\mid c\in C\}$.

To describe $EV_d( {K}\oplus a\check{C})$ as $a\to 0$, we need two integrability conditions, which we formulate here for later reference. To state them, we recall that a compact set $X\subseteq \R^d$ is called $\eps$-regular if for every $x\in \partial X$, there exist two balls $B_i,B_o \subseteq \R^d$ of radius $\eps$ such that $x\in B_i\cap B_o$, $B_i\subseteq X$ and $\indre(B_o) \subseteq \R^d \backslash X$.  

\begin{cond}\label{cond1}
The grain distribution $\Q$ satisfies $E\diam(K)^{d-1} <\infty$ and there is an $\eps>0$ such that the grains contain a.s. a ball of radius $\eps$.
\end{cond}
\begin{cond}\label{cond2}
The grain distribution $\Q$ satisfies $E\diam(K)^{d} <\infty$ and there is an $\eps>0$ such that the grains are a.s. $\eps$-regular.
\end{cond}

\begin{lem}\label{convbound}
Suppose that $C\subseteq \R^d$ is compact and $\Q$ satisfies Condition \ref{cond1}. Then there is an $M_1>0$ which is independent of $a$ such that for $a<1$,
\begin{equation*}
0\leq EV_d(K\oplus a\conv(C)) - EV_d(K\oplus aC) \leq M_1a^2.
\end{equation*}
If $\Q$ satisfies Condition \ref{cond2}, then there is an $M_2>0$ which is independent of $a$ such that for $a<1$,
\begin{equation*}
0\leq EV_d(K\oplus a\conv(C)) - EV_d(K\oplus aC) \leq M_2a^3.
\end{equation*}
\end{lem}

\begin{proof}
If $L$ is convex with twice differentiable support function and contains a ball of radius $\eps$, then \cite[Lemma 12]{kampf} shows that there is an $M_1'>0$ depending only on $d$ and $C$ such that
\begin{equation*}
0\leq V_d(L\oplus a\conv(C)) - V_d(L\oplus aC) \leq M_1'\frac{\diam(L)^{d-1}\vee 1}{\eps}a^2.
\end{equation*}
By \cite[Theorem 3.3.1]{schneider}, an arbitrary compact convex body $K$ can be approximated by a sequence $L_n$ of convex bodies with smooth support functions. We may assume that $L_n$ contains a ball of radius $\eps-\frac{1}{n}$. The map $L\mapsto V_d(L\oplus aC)$ is continuous on the space of compact convex sets with interior points, see \cite[Lemma 10]{kampf}, so by continuity of the diameter function, the same inequality holds for $L$ replaced by $K$. The assumptions of the lemma allow us to take the mean value.

Similarly, \cite[Lemma 17]{kampf} shows that if $L$ is $\eps$-regular with twice differentiable support function, then there is an $M_2'>0$ such that
\begin{equation*}
0\leq V_d(L\oplus a\conv(C)) - V_d(L\oplus aC) \leq M_2'\frac{\diam(L)^{d}\vee 1}{\eps^3}a^3.
\end{equation*}
If $K$ is $\eps$-regular we may write $K= K'\oplus \eps B^d$ \cite[Theorem 3.2.2]{schneider}
where $K'$ is also convex. Approximating $K'$ as above yields the claim in this situation as well.
\end{proof}

For convex sets $C,K\subseteq \R^d$,
\begin{equation}\label{Defn:mixedvolumes}
V_d({K}\oplus a\check{C}) = \sum_{m=0}^d \binom{d}{m}a^m V(\check{C}[m],K[d-m])
\end{equation}
with nonnegative numbers $V(\check{C}[m],K[d-m])$ which are the so-called mixed volumes, see \cite[Theorem 5.1.7]{schneider}.
The integrability condition \eqref{intcond} ensures that $EV(\check{C}[m],K[d-m])<\infty$ for all $m$.

Combining this with \eqref{exp} and Lemma \ref{convbound}, we obtain:
\begin{prop}\label{pol}
Let $C\subseteq \R^d$ be a non-empty compact set.
If the grain distribution satisfies Condition \ref{cond1}, then for $a $ sufficiently small
\begin{align*}
P(aC \subseteq \R^d \backslash Z)= e^{-\gamma EV_d(K)}(1 -  ad\gamma  EV(\conv(\check{C})[1],{K}[d-1]))+O(a^2).
\end{align*}
If the  grain distribution satisfies Condition \ref{cond2}, then for $a $ sufficiently small
\begin{align*}
 P(aC \subseteq \R^d \backslash Z) ={}& e^{-\gamma EV_d(K)} -ad\gamma e^{-\gamma EV_d(K)}EV(\conv(\check{C})[1],{K}[d-1])\\
 &- a^2e^{-\gamma EV_d(K)}\bigg(\frac{d(d-1)}{2}\gamma EV(\conv(\check{C})[2],{K}[d-2])\\
&-\frac{d^2}{2}\gamma^2 (EV(\conv(\check{C})[1],{K}[d-1]))^2\bigg)+O(a^3).
\end{align*}
\end{prop}


Next we try to obtain a more explicit expression for the mixed volumes in Proposition \ref{pol}. For convex bodies $C,K$ it is well known, see \cite[(5.19)]{schneider}, that
\begin{equation}\label{MixedVolEq}
V(\check{C}[1],K[d-1])= \frac{1}{d}\int_{S^{d-1}} h(\check{C},u) S_{d-1}(K,du).
\end{equation}
Here $S_{d-1}(K,\cdot)$ is the $(d-1)$th surface area measure of $K$ on $S^{d-1}$ and $h(C,u)=\sup \{\langle c,u \rangle, c\in C\}$ is the support function of $C$. This yields:
\begin{prop}\label{hBW}
Let $B,W\subseteq \R^d$ be two finite non-empty sets. Suppose that the grain distribution satisfies Condition \ref{cond1}. Then for $a$ sufficiently small,
\begin{align}\nonumber
\MoveEqLeft P({aB\subseteq Z  , aW \subseteq \R^d \backslash Z }) \\ \label{hit}
& =  a\gamma e^{-\gamma EV_d(K)} E\int_{S^{d-1}} (- h(B\oplus \check{W},u))^+ S_{d-1}(K, du)+ O(a^2),\\ \nonumber
\MoveEqLeft P(aB\subseteq Z  ) = 1 - e^{-\gamma EV_d(K)} + a\gamma e^{-\gamma EV_d(K)} E\int_{S^{d-1}}  h(B ,u) S_{d-1}(K, du)\\ 
&+ O(a^2).\nonumber
\end{align}
\end{prop}

Proposition \ref{hBW} is also derived in \cite[Theorem 4]{rataj} with a different approach using geometric measure theory.

\begin{proof}
We consider only the first formula. The second one is similar, only simpler.
From \eqref{EI1} and Proposition \ref{pol}, we obtain:
\begin{align*}
\MoveEqLeft P({aB\subseteq Z  , aW \subseteq \R^d \backslash Z })
= \sum_{S\subseteq B} (-1)^{|S|}P(a(S\cup W) \subseteq \R^d \backslash Z)\\
={}& e^{-\gamma EV_d(K)}\sum_{S\subseteq B} (-1)^{|S|} (1-a\gamma dEV(\conv(\check{S}\cup \check{W}) [1],K[d-1]))+O(a^2)\\
 ={}&  - a\gamma e^{-\gamma EV_d(K)} \sum_{S\subseteq B} (-1)^{|S|} E\int_{S^{d-1}}h(\check{S}\cup \check{W},u)S_{d-1}(K,du)+O(a^2).
\end{align*}

Consider a fixed $u\in S^{d-1}$ and let $B_1\subseteq B$ be the set $\{b\in B \mid -\langle b,u \rangle > h(\check{W},u)\}$. Then we may compute:
\begin{align}\nonumber
\MoveEqLeft \sum_{S\subseteq B} (-1)^{|S|} h(\check{S}\cup \check{W},u)  \\ \label{udregning}
&=
\sum_{\emptyset\neq S\subseteq B} (-1)^{|S|} \max \{h(\check{S},u),h(\check{W},u)\} +h(\check{W},u) \\\nonumber
&= \sum_{\emptyset\neq S\subseteq B, S\cap B_1 \neq \emptyset} (-1)^{|S|} h(\check{S},u) + \sum_{\emptyset\neq S\subseteq B\backslash B_1 } (-1)^{|S|} h(\check{W},u)+h(\check{W},u)  \\ \nonumber
&=  \sum_{\emptyset\neq S\subseteq B} (-1)^{|S|} h(\check{S},u)-\sum_{\emptyset\neq S\subseteq B\backslash B_1} (-1)^{|S|} h(\check{S},u) +h(\check{W},u)  \mathds{1}_{\{B =B_1\}}.
\end{align}

Using the inclusion-exclusion principle for maxima
\begin{equation*}
\max\{x_1,\dots,x_k\}= \sum_{\emptyset \neq I\subseteq \{1,\dots,k\}}(-1)^{|I|+1} \min \{x_i,i\in I\},
\end{equation*}
we find that
\begin{align*}
h(B,u)=\sum_{\emptyset \neq S\subseteq B} (-1)^{|S|} h(\check{S},u).
\end{align*}
Thus
\begin{align*}
\MoveEqLeft \sum_{\emptyset \neq S\subseteq B} (-1)^{|S|} h(\check{S},u)  -
 \sum_{\emptyset \neq S\subseteq B\backslash B_1} (-1)^{|S|} h(\check{S},u)  \\
&=h(B,u)-h(B\backslash B_1,u)\mathds{1}_{\{ B\neq B_1\} }\\
&= h(B,u)\mathds{1}_{\{ B=B_1\} }.
\end{align*}
 
The situation $B=B_1$ is equivalent to $h(B,u)<-h(\check{W},u)$, so \eqref{udregning} equals
\begin{equation*}
-(- h(B\oplus \check{W},u))^+.
\end{equation*}
\end{proof}

\begin{rem}
The formula \eqref{hit} resembles the volumes of hit-and-miss transforms in the design based setting. These are given in \cite[Theorem 5]{rataj} for a deterministic set $X$ by
\begin{align*}
\MoveEqLeft V_d(z\in \R^d \mid {z+aB\subseteq X  , z+ aW \subseteq \R^d \backslash X }) \\
&= a\int_{S^{d-1}} (- h(B\oplus \check{W},u))^+ S_{d-1}(X, du)+ O(a^2).
\end{align*}
In \eqref{hit}, $X$ is replaced by the Blaschke body $B(Z) $ associated with $Z$. This is the convex body with surface area measure $S_{d-1}(B(Z), \cdot) = \gamma ES_{d-1}(K,\cdot)$,  i.e.\ a sort of average body, see \cite[Section 4.6]{SW}. Thus, we have
\begin{align*}
\MoveEqLeft \gamma  E\int_{S^{d-1}} (- h(B\oplus \check{W},u))^+ S_{d-1}(K, du)\\
&=  \int_{S^{d-1}} (- h(B\oplus \check{W},u))^+ S_{d-1}(B(Z), du).
\end{align*}
\end{rem}

To describe $V(\check{C}[2],K[d-2])$, we introduce a bit more notation. For $\eps$-regular $K$ , let $u(x)$ be the uniquely determined outward pointing normal at $x\in \partial K$. The principal directions and principal curvatures are defined at almost all $x\in \partial K$, c.f.\ \cite{federer}, allowing us to define the second fundamental form $\II_x$. For $s\in \R^d$ we let $\II_x(s)$ denote $\II_x(\pi_x s,\pi_x s)$ where $\pi_x : \R^d \to T_x \partial K$ is the orthogonal projection.  
For a compact set  $P$, we let
\begin{align*}
\II^-_x(P){}&=\inf\{\II_x(p)\mid p\in F(\check{P},u(x)) \}\\
\II^+_x(P){}&=\sup\{\II_x(p)\mid p\in F({P},u(x)) \}.
\end{align*}
where $F(P,u)$ is the support set $\{p\in P \mid h(P,u)=\langle p,u \rangle\}$. Let $\Ha^k$ denote the $k$-dimensional Hausdorff measure.

\begin{prop}\label{Qprop}
Suppose $K\subseteq \R^d$ is a convex $\eps$-regular set and $P\subseteq \R^d$ is a convex polytope with vertex set $P_0$. Then
\begin{equation*}
{V}(\check{P}[2], {K}[d-2])= \frac{1}{d(d-1)} \int_{\partial K} ( h(\check{P}_0,u(x))^2\tr \II_x - \II_x^-({P_0}) )\Ha^{d-1}(dx).
\end{equation*}
\end{prop}

The proof below is based on \cite{am2}, but see also \cite[Theorem 4.6]{daniel}.

\begin{proof}
By \eqref{Defn:mixedvolumes} and Lemma \ref{convbound},
\begin{equation*}
V(\check{P}[2], {K}[d-2])= \frac{d^2}{da^2_+} \frac{1}{d(d-1)}V_d({K}\oplus a\check{P}) = \frac{d^2}{da_+^2}\frac{1}{d(d-1)} V_d(K\oplus a\check{P_0})
\end{equation*}
where $\frac{d^2}{da^2_+}$ is the second order right derivative at zero. A formula for $V_d\left((K \ominus a\check{B} )\backslash (K\oplus a \check{W})\right)$ where $\ominus$ is the Minkowski subtraction  is computed in \cite[Theorem 4.1]{am2}. As a special case we have that
\begin{align*}
\frac{d^2}{da_+^2} V_d\left(K\backslash (K\oplus a\check{P}_0)\right)={}& \int_{\partial K}\Big( (\II_x^-({P}_0)-h(\check{P}_0,u)^2\tr\II_x)\mathds{1}_{\{h(\check{P}_0,u)<0\}}\\
&+(\II_x^-({P}_0))^+\mathds{1}_{\{h(\check{P}_0,u)=0\}} \Big)d\Ha^{d-1}.
\end{align*}
By exactly the same line of proof as in  \cite[Theorem 4.1]{am2}, one could prove a formula for $V_d\left((K \oplus a\check{W} )\backslash (K\ominus a \check{B})\right)$. This amounts to switching the roles of $t_+(aB)$ and $t_-(aW)$ in \cite[(20)]{am2} and replacing the indicator function $\tau_{B,W}$ by $1-\tau_{B,W}$. From there, all arguments of the proof carry over. As a special case, one finds
\begin{align*}
\frac{d^2}{da_+^2} V_d\left( (K\oplus a\check{P}_0)\backslash K\right) ={}& \int_{\partial K}\Big( (h(\check{P}_0,u)^2\tr\II_x-\II_x^-({P}_0))\mathds{1}_{\{h(\check{P}_0,u)>0\}}\\
&-(\II_x^-(P_0))^- \mathds{1}_{\{h(\check{P}_0,u)=0\}} \Big)d\Ha^{d-1},
\end{align*}
and the claim follows.
\end{proof}


Writing
\begin{align*}
\MoveEqLeft Q(K,B,W)=\frac{1}{2}\int_{\partial K} \Big(((h(B,u)^2-h(\check{W},u)^2)\tr\II-\II^+(B)+\II^-(W))\\
&\times \mathds{1}_{\{h(B\oplus{\check{W}},u)<0\}} + (\II^-(W)-\II^+(B))^+\mathds{1}_{\{h(B\oplus {\check{W}},u)=0\}}\Big) d\Ha^{d-1}
\end{align*}
for simplicity, we derive:
\begin{prop}\label{propmixedvols}
Suppose $K\subseteq \R^d$ is convex $\eps$-regular and $B,W\subseteq \R^d$ are non-empty finite sets. Then
\begin{align*}
& \sum_{S\subseteq B} (-1)^{|S|} V(\conv(\check{S}\cup\check{W})[2],K[d-2]) = -\binom{d}{2}^{-1} Q(K,B,W),\\
&\sum_{\emptyset \neq S\subseteq B} (-1)^{|S|} V(\conv(\check{S})[2],K[d-2]) \\
&\qquad = -\binom{d}{2}^{-1}\int_{\partial K} (h(B,u(x))^2\tr \II_x-\II^+_x(B)) \Ha^{d-1}(dx).
\end{align*}
\end{prop}

\begin{rem}\label{probo2}
By Equation \eqref{EI1}, Proposition \ref{pol}, \eqref{MixedVolEq} and Proposition \ref{propmixedvols},
\begin{align}\nonumber
\MoveEqLeft\frac{d^2}{da^2_+} P(aB\subseteq Z, aW \subseteq \R^d \backslash Z) =e^{-\gamma EV_d(K)}\bigg( 2\gamma  Q(K,B,W) \\
&+\sum_{S\subseteq B} (-1)^{|S|}  \gamma^2\bigg(E\int_{S^{d-1}}h(\check{S}\cup \check{W},u)S_{d-1}(K,du)\bigg)^2\bigg).\label{secondderiv}
\end{align}
The first term is similar to what we see for a deterministic set \cite[Theorem 4.1]{am2}, whereas the second term is new and must originate from the underlying distribution. This is, however, desirable, since it corresponds to the second term in the formula for $\altoverline{V}_{d-2}(Z)$ in \eqref{specific}. The sum in \eqref{secondderiv} does not seem to reduce to anything simple. In particular, Table \ref{Qer} shows that it does not need to vanish if $h(B\oplus \check{W},u)\geq 0$ for all $u\in S^{d-1}$, that is, if $(B,W)$ cannot be separated by a hyperplane. This is very different from the design based setting where such configurations do not contribute to the second order formulas. It is a consequence of the fact that we allow grains to overlap in the Boolean model, otherwise such configurations would not occur for sufficiently small $a$.
\end{rem}

\begin{proof}
We only consider the first equality. The second is shown similarly.
By Proposition \ref{Qprop} we must consider
\begin{align*}
\sum_{S\subseteq B} (-1)^{|S|} \int_{\partial K} ( h(\check{S}\cup\check{W},u(x))^2\tr \II_x - \II_x^-(S\cup W) )\Ha^{d-1}(dx).
\end{align*}

The same argument as in the proof of Proposition \ref{hBW}, now using the relation
\begin{equation*}
\max\{x_1,\dots,x_k\}^2= \sum_{\emptyset \neq I\subseteq \{1,\dots,k\}}(-1)^{|I|+1} \min \{x_i,i\in I\}^2,
\end{equation*}
shows that 
\begin{align*}
\sum_{S\subseteq B} (-1)^{|S|+1} h(\check{S}\cup \check{W},u)^2= (h(B,u)^2-h(\check{W},u)^2)\mathds{1}_{\{h(B\oplus \check{W},u)<0\}}.
\end{align*}

Fix $x\in \partial K$ and let $u=u(x)$. Write $B$ as a disjoint union $B_1 \cup \dots \cup B_k$ of non-empty sets such that there are real numbers $s_1> \dots > s_k$ with $\langle b,u\rangle = s_i$ for all $b\in B_i$. Then
\begin{align*}
\sum_{S\subseteq B} (-1)^{|S|} \II^-_x({S} \cup {W})={}& \II^-_x( {W}) +\sum_{m=1}^k\prod_{i=1}^{m-1} \bigg(\sum_{S_i\subseteq B_i} (-1)^{|S_i|} \bigg)\\
&\times \sum_{\emptyset \neq S_m\subseteq B_m}(-1)^{|S_m|} \II^-_x( S_m \cup {W}).
\end{align*} 
Note that all terms with $m>1$ vanish because $\sum_{S_1\subseteq B_1} (-1)^{|S_1|} =0$. Hence
\begin{equation*}
\sum_{S\subseteq B} (-1)^{|S|} \II^-_x({S}\cup {W})=  \sum_{ S\subseteq B_1}(-1)^{|S|}\II^-_x( S \cup {W}).
\end{equation*}   
There are now three possibilities: $h(\check{W},u)<-h(B,u)$, $h(\check{W},u)>-h(B,u)$, and $h(\check{W},u)=-h(B,u)$.

 The first inequality means that $F(\check{B}_1\cup \check{W},u)=F( \check{B}_1,u)$. In this case:
 \begin{align*}
 \sum_{ S\subseteq B_1}(-1)^{|S|}\II^-_x( S \cup {W}) &= \II^-_x( {W})-\sum_{\emptyset \neq S\subseteq B_1}(-1)^{|S|+1}\II^-_x( S )\\
 &= \II^-_x( {W})-\max \{\II_x(b) \mid b\in B_1\}\\
 &=\II^-_x( {W})-\II^+_x(B).
\end{align*}

In the second case, $F(\check{B}_1\cup \check{W},u)=F(\check{W},u)$. Hence
\begin{equation*}
\sum_{S\subseteq B} (-1)^{|S|} \II^-_x({S}\cup {W})=  \sum_{ S\subseteq B_1}(-1)^{|S|}\II^-_x( {W})=0.
\end{equation*}   
  
For the third case, let $B_{1}^0=\{b\in B_{1}\mid \II_x(b)\leq \II^-_x({W})\}$.
Then
\begin{align*}
 \MoveEqLeft \sum_{S\subseteq B_1} (-1)^{|S|} \II^-_x({S}\cup {W})\\
&= 
\sum_{S\subseteq B_1\backslash B_1^0} (-1)^{|S|} \II^-_x( {W}) + \sum_{S\cap B_1^0 \neq \emptyset } (-1)^{|S|} \II^-_x( S)\\
&= \II^-_x( {W})\mathds{1}_{\{B_1=B_1^0\}}+ \sum_{\emptyset \neq S\subseteq B_1 } (-1)^{|S|} \II^-_x( {S})-\sum_{\emptyset \neq S\subseteq B_1\backslash B_1^0} (-1)^{|S|} \II^-_x({S})\\
 &= \II^-_x( {W})\mathds{1}_{\{B_1=B_1^0\}}-\II^+_x(B) + \II^+_x(B)\mathds{1}_{\{B_1\neq B_1^0\}}\\
 &= (\II^-_x( {W})-\II^+_x(B))^+,
\end{align*}
since $B_1=B_1^0$ is equivalent to $\II^+_x(B)\leq \II^-_x( {W})$.
\end{proof}

In many cases, the expression for $Q(K,B,W)$ can be simplified, since:

\begin{proof}
The set $\{x\in \partial K \mid h(B\oplus \check{W},u(x))=0\}$ is contained in the union
\begin{equation*}
\bigcup_{b\in B,w\in W} D_{b,w}
\end{equation*} 
where $D_{b,w}=\{x\in \partial K \mid \langle b-w,u(x)\rangle=0\}$.

Let $b\in B$ and $w\in W$ be fixed. The function $g: \partial K \to \R$ given by $g(x)=\langle b-w,u(x)\rangle$ is almost everywhere continuously differentiable, see \cite{federer}. 
 A critical point of $g$ is a point $x\in \partial K$ with \ $dg_x(v)=\langle b-w,du_x(v)\rangle=0$ for all $v\in T_x\partial K=u(x)^\perp$.

If $\partial K$ is $C^2$,
the implicit function theorem says that every non-critical point of $g$ in $g^{-1}(0)=D_{b,w}$ has a neighborhood in which $g^{-1}(0)$ constitutes a $(d-2)$-dimensional $C^1$-manifold. 
Thus, it follows that the set of non-critical points of $g$ in $D_{b,w}$ has $\Ha^{d-1}$-measure 0.

Suppose that  $x\in D_{b,w}$, $\II^-_x(W)=\II_x(w)$, $\II^+_x(B)=\II_x(b)$, and that $x$ is a critical point of $g$. 
Then either $b=w$ or $b-w$ is a principal direction at $x$ with principal curvature $0$. Hence 
\begin{equation}\label{IIeq}
\II_x(b)-\II_x(w)=\II_x(\pi(b))-\II_x(\pi(w))=0
\end{equation}
 where $\pi$ is the projection onto $(b-w)^\perp\cap T_x \partial K$ so that $\pi(b)=\pi(w)$. Hence $\II^-_x (W)=\II^+_x(B)$.

In the convex case, $D_{b,w}$ is contained in the boundary of the cylinder $\pi_{{(b-w)}^{\perp}}(K)\times \text{span}(b-w)$, where $\pi_{{(b-w)}^{\perp}}: \R^d \to (b-w)^\perp$ is the projection. Clearly, any $x\in D_{b,w}$ is either the only point on the line through $x$ parallel to $b-w$, or $b-w$ is a principal direction at $x$ with principal curvature $0$. Thus we can use Equation \eqref{IIeq} above to obtain
\begin{align*}
\MoveEqLeft \Ha^{d-1}(D_{b,w}\cap \{\II(w)\neq \II (b)\})\\
& = \int_{ \pi_{{(b-w)}^{\perp}}(D_{b,w})} \int_{\spa(b-w)}\mathds{1}_{\partial K}(x+y) \mathds{1}_{\{\II_{x+y}(b)\neq \II_{x+y}(w)\} } dx \Ha^{d-1}(dy)\\
&=0.
\end{align*}

\end{proof}

\section{Applications to digital images}\label{dig}
In this section we introduce our model for digital images and define local algorithms. We then apply the formulas of Section \ref{hitmiss} to determine their mean values when applied to Boolean models.

\subsection{Local algorithms}
Let $\La$ be a lattice in $\R^d$ spanned by linearly independent vectors $v_1,\dots,v_d$. We denote by $C_0^n$ the $n\times \dotsm \times n$ fundamental cell $C_0^n=\bigoplus_{i}[0,nv_i) $ and by $C_{0,0}^n=C_0^n\cap \La$ the set of lattice points lying in this set. Their respective translations by $z\in \R^d$ are denoted by $C_z^n=z+C_0^n$ and $C_{z,0}^n=z+C_{0,0}^n$. 

Let $Z$ be a stationary Boolean model and consider a digital black-and-white image of $Z$ in a compact convex observation window $A$. This is modeled as $Z\cap A\cap \La$. We change the resolution by multiplying $\La$ by a factor $a>0$.
From the information $Z\cap A\cap a\La$, we want to estimate the specific intrinsic volumes $\altoverline{V}_q(Z)$. A so-called local algorithm for this is defined as follows: 

Consider the set of $n\times \dots \times n$ configurations. These are pairs $(B,W)$ with $B\cup W=C_{0,0}^n$ and $B\cap W=\emptyset$.
We enumerate the elements of $C_{0,0}^n$ in the following way. For $x=\sum_{k=1}^d \lambda_kv_k \in C_{0,0}^n$ with $\lambda_k \in \{0,\dots, n-1\}$ write $x=x_i$ where
\begin{equation*}
i=\sum_{k=1}^d \lambda_k n^{k-1}.
\end{equation*}
 There are $2^{n^d}$ possible configurations. We denote these by $(B_l,W_l)$, $l=0,\dots,2^{n^d}-1$, where
\begin{equation*}
l=\sum_{i=0}^{{n^d}-1} 2^{i} \mathds{1}_{\{x_i\in B\}}.
\end{equation*}

A local algorithm for $\altoverline{V}_q$ is an algorithm of the form
\begin{align}\label{loces}
\hat{V}_q(Z\cap A){}&= a^{q-d} \sum_{l = 0}^{2^{n^d}-1} w_l^{(q)} \frac{N_l(Z\cap A\cap a\La)}{N(A)}
\end{align}
where 
\begin{align}\label{Nest}
N_l(Z\cap A \cap a\La){}&=\sum_{z \in a\La \cap (A \ominus a\check{C}_{0,0}^n)} \mathds{1}{\{z+aB_l\subseteq Z, z + aW_l \subseteq \R^d \backslash Z\}}
\end{align}
is the number of occurrences of the configuration $(B_l,W_l)$ inside $A$. This is weighted by the weight $w_l^{(q)}\in \R$.
Moreover, $A\ominus \check{C}_{0,0}^n=\{x\in \R^d \mid x + C_{0,0}^n \subseteq A\}$, and
$N(A)$ denotes the cardinality of $a\La \cap (A \ominus a\check{C}_{0,0}^n)$, i.e.\ the total number of configurations in $A$.

Recall that in the definition of specific intrinsic volumes \eqref{specdef} and \eqref{specdef2} we remove the contribution to ${V}_q(Z\cap A)$ coming from the boundary of the observation window. For this reason, we count in \eqref{Nest} only configurations lying completely in the interior of $A$. 

The mean value of \eqref{Nest} is
\begin{align*}
EN_l(Z\cap A \cap a\La) {}&= \sum_{z \in a\La \cap (A \ominus a\check{C}_{0,0}^n) } P({z+ aB_l\subseteq Z , z+ aW_l \subseteq \R^d \backslash Z })\\
 &=N(A) P({aB_l\subseteq Z , aW_l \subseteq \R^d \backslash Z }),
\end{align*}
and hence
\begin{equation}\label{meanHM}
E\hat{V}_{q}(Z\cap A)
 =a^{q-d}\sum_{l=0}^{2^{n^d}-1} w_l^{(q)}P({aB_l\subseteq Z,  aW_l \subseteq \R^d \backslash Z }).
\end{equation}

\subsection{Asymptotic formulas for the mean digital estimators}
The formulas of Section \ref{hitmiss} yield asymptotic expressions for \eqref{meanHM} as the grid width $a$ goes to zero. First we consider estimators for the specific volume $\hat{V}_d(Z\cap A)$.

\begin{thm}
For any Boolean model,
\begin{equation*}
E\hat {V}_d(Z\cap A) = w_0^{(d)} e^{-\gamma E V_d(K)} + w_{2^{n^d}-1}^{(d)}(1-e^{-\gamma E V_d(K)}) + O(a).
\end{equation*}
In particular, $\hat {V}_d$ is asymptotically unbiased iff $w_0^{(d)}=0$ and $w_{2^{n^d}-1}^{(d)}=1$.
\end{thm}
\begin{proof}
The result follows from an application of Proposition \ref{pol} and Proposition \ref{hBW}
to \eqref{meanHM}.
\end{proof}

\begin{rem}In fact, it is well known that the estimator based on $1\times \dotsm \times 1$ configurations with $w_0^{(d)}=0$ and $w_{1}^{(d)}=1$ is unbiased, even for fixed $a$. This is the natural estimator given by counting lattice points in $Z\cap A$. Hence we will not discuss volume estimation further.
\end{rem}

Next we consider surface estimators.
\begin{thm} \label{mean1st}
For any stationary Boolean model satisfying Condition \ref{cond1},
\begin{equation*}
\lim_{a\to 0 } E\hat{V}_{d-1}(Z\cap A) 
\end{equation*}
exists  if and only if $w_0^{(d-1)}=w^{(d-1)}_{2^{n^d}-1}=0$. 

In this case, 
\begin{equation*}
 E\hat{V}_{d-1}(Z\cap A) = \gamma e^{-\gamma EV_d(K)} \sum_{l=1}^{2^{n^d}-2} w_l^{(d-1)} E\int_{S^{d-1}} (-h(B_l\oplus \check{W}_l,u))^+ S_{d-1} (K,du) + O(a).
\end{equation*}
 
 If Condition \ref{cond2} is satisfied, 
 \begin{align*}
\MoveEqLeft E\hat{V}_{d-1}(Z\cap A) -  \lim_{a\to 0} E\hat{V}_{d-1}(Z\cap A)=  ae^{-\gamma EV_d(K)} \sum_{l=1}^{2^{n^d}-2} w_l^{(d-1)} \bigg(\gamma E Q(K,B_l,W_l)\\
{}&+ \frac{\gamma^2}{2}\sum_{S\subseteq B_l}(-1)^{|S|} \bigg(E\int_{S^{d-1}} h(\check{S}\cup \check{W}_l,u)  S_{d-1}(K,du)\bigg)^2 \bigg)+ O(a^2).
 \end{align*}
\end{thm}
\begin{proof}
Under Condition \ref{cond1} the result follows by applying Proposition \ref{pol} and Proposition \ref{hBW} to \eqref{meanHM}. Under Condition \ref{cond2} we use additionally Remark \ref{probo2}.
\end{proof}
Finally we consider estimators for the integrated mean curvature.
\begin{thm}\label{mean2nd}
For any stationary Boolean model satisfying Condition \ref{cond2},
\begin{equation*}
\lim_{a\to 0 } E\hat{V}_{d-2}(Z\cap A) 
\end{equation*}
exists  if and only if $w_0^{(d-2)}=w^{(d-2)}_{2^{n^d}-1}=0$ and  
\begin{equation}\label{cond3}
\sum_{l=1}^{2^{n^d}-2} w_l^{(d-2)}E\int_{S^{d-1}} (-h(B_l\oplus \check{W}_l,u))^+ S_{d-1}(K,du) =0.
\end{equation}
 
In this case,
 \begin{align*}
 &E \hat{V}_{d-2} (Z\cap A)= e^{-\gamma EV_d(K)} \sum_{l=1}^{2^{n^d}-2} w_l^{(d-2)} \bigg(\gamma EQ(K,B_l,W_l)\\
&+ \frac{\gamma^2}{2}\sum_{S\subseteq B_l} (-1)^{|S|}\bigg(E\int_{S^{d-1}} h(\check{S}\cup \check{W}_l,u)  dS_{d-1}(K,du)\bigg)^2\bigg) + O(a).
 \end{align*}
\end{thm}
\begin{proof}
The statement is obtained by applying Proposition \ref{pol}, Proposition \ref{hBW} and  Remark \ref{probo2} to \eqref{meanHM}. 
\end{proof}
We obtain the following corollary.

\begin{cor}
There exists no local estimator based on $n\times \dotsm \times n$ configurations for $\altoverline{V}_{d-1}(Z)$ if $d\geq 2$ or for $\altoverline{V}_{d-2}(Z)$   if $d\geq 3$ such  that it is asymptotically unbiased for all stationary Boolean models satisfying Condition \ref{cond1} or \ref{cond2}, respectively.
\end{cor}

\begin{proof}
We consider a Boolean model with a fixed grain equal to some convex body $K_0$.
By $\hat{V}_{d-1}(K_0)$ we mean the digital estimator of $V_{d-1}(K_0)$ in the designed based setting (i.e. based on a stationary random lattice) with the same weights as in the definition of $\hat{V}_{d-1}(Z\cap A)$.
Then Theorem \ref{mean1st} and \cite[Theorem 4.1]{am3} (or originally \cite[Theorem 5]{rataj}) imply
\[\lim\limits_{a\to 0} E \hat{V}_{d-1}(Z\cap A)= \gamma e^{\gamma E V_d(K_0)} \lim\limits_{a\to 0} E \hat{V}_{d-1}(K_0). \]
By \eqref{specific} the estimator $ \hat{V}_{d-1}(Z\cap A)$ is asymptotically unbiased if $\lim\limits_{a\to 0} E \hat{V}_{d-1}(K_0)= V_{d-1}(K_0).$ This is not the case if we choose $K_0$ as one of the counterexamples in \cite[Theorem 1.4]{am3}. Note that the counterexamples are chosen convex in the proof. 

In the same way denote by $\hat{V}_{d-2}(K_0)$ the digital estimator of $V_{d-2}(K_0)$ in the design based setting with the same weights as in the definition of $\hat{V}_{d-2}(Z\cap A)$.  
Then, by \cite[Theorem 4.2]{am3} (originally shown in \cite{am2}) it holds  
\begin{align*}
&\lim\limits_{a\to 0} E \hat{V}_{d-2}(Z\cap A) = e^{-\gamma V_d(K_0)}  \bigg(\gamma 
\lim\limits_{a\to 0} E \hat{V}_{d-2}(K_0)\\
&+ \frac{\gamma^2}{2}\sum_{l=1}^{2^{n^d}-2} w_l^{(d-2)} \sum_{S\subseteq B_l} (-1)^{|S|}\bigg(\int_{S^{d-1}} h(\check{S}\cup \check{W}_l,u)  dS_{d-1}(K_0,du)\bigg)^2\bigg)
\end{align*}
Comparing the coefficient of $\gamma$ with the one in the corresponding formula in \eqref{specific} we obtain that the estimator $\hat{V}_{d-2}(Z\cap A)$ can only be asymptotically unbiased if $\lim\limits_{a\to 0} E\hat{V}_{d-2}(K_0)=V_{d-2}(K_0)$. Again this is not the case if we choose $K_0$ as one of the counterexamples in \cite[Theorem 1.4]{am3}. This yields the assertion.
\end{proof}

\section{Optimal estimators\\ for isotropic Boolean models}\label{isotropy}
We now specialise to the case where $Z$ is stationary and the grain distribution $\Q$ is rotation invariant. 

\begin{thm}\label{mean1stIsotropic}
Let $Z$ be a stationary, isotropic Boolean model.
If Condition \ref{cond1}  is satisfied and $w_0^{(d-1)}=w_{2^{n^d}-1}^{(d-1)}=0$, then
\begin{align*}
E \hat{V}_{d-1}(Z\cap A)&=\gamma e^{-\gamma EV_d(K)}\sum\limits_{l=1}^{2^{n^d}-2}
w_l^{(d-1)}c_1(B_l,W_l) EV_{d-1}(K)+O(a),
\end{align*}
where $c_1(B_l,W_l)$ is a constant.
If Condition \ref{cond2}, $w_0^{(d-2)}=w_{2^{n^d}-1}^{(d-2)}=0$ and \eqref{cond3}
are satisfied, then
\begin{align*}
E\hat{V}_{d-2}(Z\cap A)&= e^{-\gamma E V_d(K)}
\sum\limits_{l=1}^{2^{n^d}-2}
w_l^{(d-2)}\Big(\gamma c_2(B_l,W_l)E V_{d-2}(K)\\
&\quad+\frac{\gamma^2}{2}c_3(B_l,W_l)\left(  E V_{d-1}(K)\right)^2\Big)+O(a),
\end{align*}
where $c_2(B_l,W_l)$ and $c_3(B_l,W_l)$ are constants.
\end{thm}
\begin{proof}

Let $l\in\{1,\ldots,2^{n^d}-2\}.$ Then by Tonelli's theorem
\begin{align*}
\MoveEqLeft E \int_{S^{d-1}} (-h(B_l\oplus \check{W}_l,u) )^+S_{d-1}(K,du)\\
{}& = E \int_{SO(d)} \int_{S^{d-1}} (-h(B_l\oplus \check{W}_l,u) )^+S_{d-1}(RK,du)dR\\
& = 2EV_{d-1}(K) (d\kappa_d)^{-1} \int_{S^{d-1}}  (-h(B_l\oplus \check{W}_l,u) )^+\Ha^{d-1}(du)\\
& =c_1(B_l,W_l) EV_{d-1}(K)
\end{align*}
where $c_1(B_l,W_l)$ is a constant.
By Fubini's theorem and \cite[Section 5]{am2} 
\begin{align*}
EQ(K,B_l,W_l)=c_2(B_l,W_l) E V_{d-2}(K)  
\end{align*}
where $c_2(B_l,W_l)$ is a constant.
Similarly,
\begin{align*}
\MoveEqLeft \sum_{S\subseteq B_l  } (-1)^{|S|} \bigg( E\int_{S^{d-1}} h(\check{S} \cup \check{W}_l,u) S_{d-1}(K,du) \bigg)^2 \\
&= 
4\left(EV_{d-1}(K)\right)^2 (d\kappa_d)^{-2}\sum_{S\subseteq B_l  } (-1)^{|S|} \bigg( \int_{S^{d-1}} h(\check{S} \cup \check{W}_l,u) \Ha^{d-1}(du) \bigg)^2  \\
&= c_{3}(B_l,W_l)\left(EV_{d-1}(K)\right)^2
\end{align*}
where $c_{3}(B_l,W_l)$ is a constant.
Inserting this in Theorem \ref{mean1st} and \ref{mean2nd} yields the assertion.
\end{proof}

Comparing Theorem \ref{mean1stIsotropic} with the Miles formulas \eqref{specific} we see that an estimator for $\overline{V}_{d-1}(Z)$ or $\overline{V}_{d-2}(Z)$ is asymptotically unbiased exactly if the weights satisfy a set of linear equations involving the constants $c_{k}(B_l,W_l)$, $k=1,2,3$. In 2D these equations were determined and the full solution was given in \cite{am}. In the following sections, we determine the constants and the corresponding equations in 3D. 
\
\subsection{The 3D situation}
For the remainder of this section we specialise to the situation $d=3$ and to $2\times 2 \times 2$ configurations on a square grid $\Z^3\subseteq \R^3$.

 Let $R$ be a rigid motion. If $RS=S'$ then $P(aS\subseteq \R^3\backslash Z)=P(aS' \subseteq \R^3\backslash Z)$. 
Thus, the isotropy allows us to reduce the number of configurations in the following way.
There are 22 motion equivalence classes of subsets of $C_{0,0}^2$. We denote these by $\eta_j$. Let $\eta_{22}=\{\emptyset \}$ and for $j\neq 22$ let $\eta_j$ be the class with the corresponding set of white points in Figure \ref{eta} as representative.
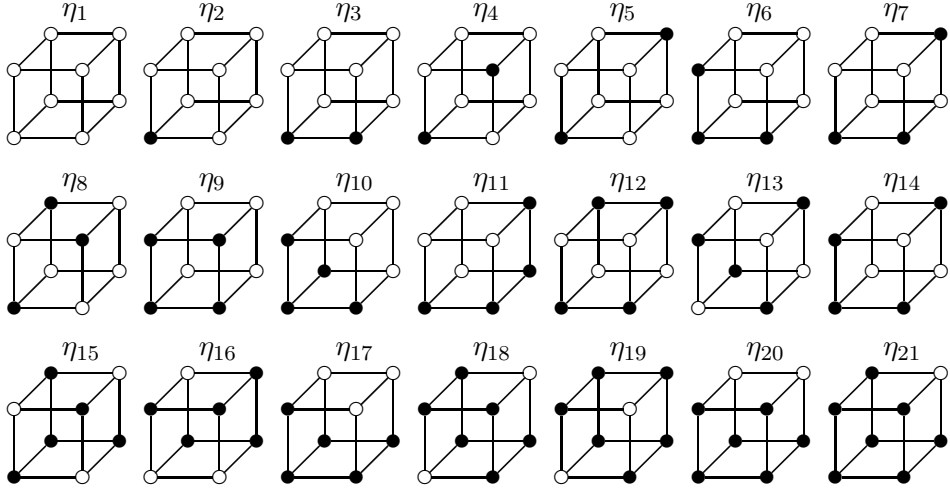
\begin{figure}[h!]
\begin{equation*}
\setlength{\unitlength}{0.09cm}
\begin{picture}(140,65)
\put(7,68){$\eta_{1}$}
\put(27,68){$\eta_{2}$}
\put(47,68){$\eta_{3}$}
\put(67,68){$\eta_{4}$}
\put(87,68){$\eta_{5}$}
\put(107,68){$\eta_{6}$}
\put(127,68){$\eta_{7}$}

\put(7,43){$\eta_8$}
\put(27,43){$\eta_{9}$}
\put(47,43){$\eta_{10}$}
\put(67,43){$\eta_{11}$}
\put(87,43){$\eta_{12}$}
\put(107,43){$\eta_{13}$}
\put(127,43){$\eta_{14}$}

\put(7,18){$\eta_{15}$}
\put(27,18){$\eta_{16}$}
\put(47,18){$\eta_{17}$}
\put(67,18){$\eta_{18}$}
\put(87,18){$\eta_{19}$}
\put(107,18){$\eta_{20}$}
\put(127,18){$\eta_{21}$}

\put(0,25){\circle*{2}}
\put(10,25){\circle{2}}
\put(0,35){\circle{2}}
\put(10,35){\circle*{2}}
\put(5.4,30.4){\circle{2}}
\put(15.4,30.4){\circle{2}}
\put(5.4,40.4){\circle*{2}}
\put(15.4,40.4){\circle{2}}
\put(0,26){\line(0,1){8}}
\put(1,25){\line(1,0){8}}
\put(10,26){\line(0,1){8}}
\put(1,35){\line(1,0){8}}

\put(5.4,31.4){\line(0,1){8}}
\put(6.4,30.4){\line(1,0){8}}
\put(15.4,31.4){\line(0,1){8}}
\put(6.4,40.4){\line(1,0){8}}

\put(0.7,25.7){\line(1,1){4}}
\put(10.7,25.7){\line(1,1){4}}
\put(0.7,35.7){\line(1,1){4}}
\put(10.7,35.7){\line(1,1){4}}
\put(20,25){\circle*{2}}
\put(30,25){\circle*{2}}
\put(20,35){\circle*{2}}
\put(30,35){\circle*{2}}
\put(25.4,30.4){\circle{2}}
\put(35.4,30.4){\circle{2}}
\put(25.4,40.4){\circle{2}}
\put(35.4,40.4){\circle{2}}
\put(20,26){\line(0,1){8}}
\put(21,25){\line(1,0){8}}
\put(30,26){\line(0,1){8}}
\put(21,35){\line(1,0){8}}

\put(25.4,31.4){\line(0,1){8}}
\put(26.4,30.4){\line(1,0){8}}
\put(35.4,31.4){\line(0,1){8}}
\put(26.4,40.4){\line(1,0){8}}

\put(20.7,25.7){\line(1,1){4}}
\put(30.7,25.7){\line(1,1){4}}
\put(20.7,35.7){\line(1,1){4}}
\put(30.7,35.7){\line(1,1){4}}

\put(40,25){\circle*{2}}
\put(50,25){\circle*{2}}
\put(40,35){\circle*{2}}
\put(50,35){\circle{2}}
\put(45.4,30.4){\circle*{2}}
\put(55.4,30.4){\circle{2}}
\put(45.4,40.4){\circle{2}}
\put(55.4,40.4){\circle{2}}
\put(40,26){\line(0,1){8}}
\put(41,25){\line(1,0){8}}
\put(50,26){\line(0,1){8}}
\put(41,35){\line(1,0){8}}

\put(45.4,31.4){\line(0,1){8}}
\put(46.4,30.4){\line(1,0){8}}
\put(55.4,31.4){\line(0,1){8}}
\put(46.4,40.4){\line(1,0){8}}

\put(40.7,25.7){\line(1,1){4}}
\put(50.7,25.7){\line(1,1){4}}
\put(40.7,35.7){\line(1,1){4}}
\put(50.7,35.7){\line(1,1){4}}

\put(60,25){\circle*{2}}
\put(70,25){\circle*{2}}
\put(60,35){\circle{2}}
\put(70,35){\circle{2}}
\put(65.4,30.4){\circle{2}}
\put(75.4,30.4){\circle*{2}}
\put(65.4,40.4){\circle{2}}
\put(75.4,40.4){\circle*{2}}
\put(60,26){\line(0,1){8}}
\put(61,25){\line(1,0){8}}
\put(70,26){\line(0,1){8}}
\put(61,35){\line(1,0){8}}

\put(65.4,31.4){\line(0,1){8}}
\put(66.4,30.4){\line(1,0){8}}
\put(75.4,31.4){\line(0,1){8}}
\put(66.4,40.4){\line(1,0){8}}

\put(60.7,25.7){\line(1,1){4}}
\put(70.7,25.7){\line(1,1){4}}
\put(60.7,35.7){\line(1,1){4}}
\put(70.7,35.7){\line(1,1){4}}

\put(80,25){\circle*{2}}
\put(90,25){\circle*{2}}
\put(80,35){\circle{2}}
\put(90,35){\circle{2}}
\put(85.4,30.4){\circle{2}}
\put(95.4,30.4){\circle{2}}
\put(85.4,40.4){\circle*{2}}
\put(95.4,40.4){\circle*{2}}
\put(80,26){\line(0,1){8}}
\put(81,25){\line(1,0){8}}
\put(90,26){\line(0,1){8}}
\put(81,35){\line(1,0){8}}

\put(85.4,31.4){\line(0,1){8}}
\put(86.4,30.4){\line(1,0){8}}
\put(95.4,31.4){\line(0,1){8}}
\put(86.4,40.4){\line(1,0){8}}

\put(80.7,25.7){\line(1,1){4}}
\put(90.7,25.7){\line(1,1){4}}
\put(80.7,35.7){\line(1,1){4}}
\put(90.7,35.7){\line(1,1){4}}

\put(100,25){\circle{2}}
\put(110,25){\circle*{2}}
\put(100,35){\circle*{2}}
\put(110,35){\circle{2}}
\put(105.4,30.4){\circle*{2}}
\put(115.4,30.4){\circle{2}}
\put(105.4,40.4){\circle{2}}
\put(115.4,40.4){\circle*{2}}
\put(100,26){\line(0,1){8}}
\put(101,25){\line(1,0){8}}
\put(110,26){\line(0,1){8}}
\put(101,35){\line(1,0){8}}

\put(105.4,31.4){\line(0,1){8}}
\put(106.4,30.4){\line(1,0){8}}
\put(115.4,31.4){\line(0,1){8}}
\put(106.4,40.4){\line(1,0){8}}

\put(100.7,25.7){\line(1,1){4}}
\put(110.7,25.7){\line(1,1){4}}
\put(100.7,35.7){\line(1,1){4}}
\put(110.7,35.7){\line(1,1){4}}

\put(120,25){\circle*{2}}
\put(130,25){\circle*{2}}
\put(120,35){\circle*{2}}
\put(130,35){\circle{2}}
\put(125.4,30.4){\circle{2}}
\put(135.4,30.4){\circle{2}}
\put(125.4,40.4){\circle{2}}
\put(135.4,40.4){\circle*{2}}
\put(120,26){\line(0,1){8}}
\put(121,25){\line(1,0){8}}
\put(130,26){\line(0,1){8}}
\put(121,35){\line(1,0){8}}

\put(125.4,31.4){\line(0,1){8}}
\put(126.4,30.4){\line(1,0){8}}
\put(135.4,31.4){\line(0,1){8}}
\put(126.4,40.4){\line(1,0){8}}

\put(120.7,25.7){\line(1,1){4}}
\put(130.7,25.7){\line(1,1){4}}
\put(120.7,35.7){\line(1,1){4}}
\put(130.7,35.7){\line(1,1){4}}

\put(0,0){\circle*{2}}
\put(10,0){\circle{2}}
\put(0,10){\circle{2}}
\put(10,10){\circle*{2}}
\put(5.4,5.4){\circle*{2}}
\put(15.4,5.4){\circle*{2}}
\put(5.4,15.4){\circle*{2}}
\put(15.4,15.4){\circle{2}}
\put(0,1){\line(0,1){8}}
\put(1,0){\line(1,0){8}}
\put(10,1){\line(0,1){8}}
\put(1,10){\line(1,0){8}}

\put(5.4,6.4){\line(0,1){8}}
\put(6.4,5.4){\line(1,0){8}}
\put(15.4,6.4){\line(0,1){8}}
\put(6.4,15.4){\line(1,0){8}}

\put(0.7,0.7){\line(1,1){4}}
\put(10.7,0.7){\line(1,1){4}}
\put(0.7,10.7){\line(1,1){4}}
\put(10.7,10.7){\line(1,1){4}}

\put(20,0){\circle{2}}
\put(30,0){\circle{2}}
\put(20,10){\circle*{2}}
\put(30,10){\circle*{2}}
\put(25.4,5.4){\circle*{2}}
\put(35.4,5.4){\circle*{2}}
\put(25.4,15.4){\circle{2}}
\put(35.4,15.4){\circle*{2}}
\put(20,1){\line(0,1){8}}
\put(21,0){\line(1,0){8}}
\put(30,1){\line(0,1){8}}
\put(21,10){\line(1,0){8}}

\put(25.4,6.4){\line(0,1){8}}
\put(26.4,5.4){\line(1,0){8}}
\put(35.4,6.4){\line(0,1){8}}
\put(26.4,15.4){\line(1,0){8}}

\put(20.7,0.7){\line(1,1){4}}
\put(30.7,0.7){\line(1,1){4}}
\put(20.7,10.7){\line(1,1){4}}
\put(30.7,10.7){\line(1,1){4}}

\put(40,0){\circle*{2}}
\put(50,0){\circle*{2}}
\put(40,10){\circle*{2}}
\put(50,10){\circle{2}}
\put(45.4,5.4){\circle*{2}}
\put(55.4,5.4){\circle*{2}}
\put(45.4,15.4){\circle{2}}
\put(55.4,15.4){\circle{2}}
\put(40,1){\line(0,1){8}}
\put(41,0){\line(1,0){8}}
\put(50,1){\line(0,1){8}}
\put(41,10){\line(1,0){8}}

\put(45.4,6.4){\line(0,1){8}}
\put(46.4,5.4){\line(1,0){8}}
\put(55.4,6.4){\line(0,1){8}}
\put(46.4,15.4){\line(1,0){8}}

\put(40.7,0.7){\line(1,1){4}}
\put(50.7,0.7){\line(1,1){4}}
\put(40.7,10.7){\line(1,1){4}}
\put(50.7,10.7){\line(1,1){4}}

\put(60,0){\circle{2}}
\put(70,0){\circle*{2}}
\put(60,10){\circle*{2}}
\put(70,10){\circle*{2}}
\put(65.4,5.4){\circle*{2}}
\put(75.4,5.4){\circle*{2}}
\put(65.4,15.4){\circle*{2}}
\put(75.4,15.4){\circle{2}}
\put(60,1){\line(0,1){8}}
\put(61,0){\line(1,0){8}}
\put(70,1){\line(0,1){8}}
\put(61,10){\line(1,0){8}}

\put(65.4,6.4){\line(0,1){8}}
\put(66.4,5.4){\line(1,0){8}}
\put(75.4,6.4){\line(0,1){8}}
\put(66.4,15.4){\line(1,0){8}}

\put(60.7,0.7){\line(1,1){4}}
\put(70.7,0.7){\line(1,1){4}}
\put(60.7,10.7){\line(1,1){4}}
\put(70.7,10.7){\line(1,1){4}}

\put(80,0){\circle{2}}
\put(90,0){\circle*{2}}
\put(80,10){\circle*{2}}
\put(90,10){\circle{2}}
\put(85.4,5.4){\circle*{2}}
\put(95.4,5.4){\circle*{2}}
\put(85.4,15.4){\circle*{2}}
\put(95.4,15.4){\circle*{2}}
\put(80,1){\line(0,1){8}}
\put(81,0){\line(1,0){8}}
\put(90,1){\line(0,1){8}}
\put(81,10){\line(1,0){8}}

\put(85.4,6.4){\line(0,1){8}}
\put(86.4,5.4){\line(1,0){8}}
\put(95.4,6.4){\line(0,1){8}}
\put(86.4,15.4){\line(1,0){8}}

\put(80.7,0.7){\line(1,1){4}}
\put(90.7,0.7){\line(1,1){4}}
\put(80.7,10.7){\line(1,1){4}}
\put(90.7,10.7){\line(1,1){4}}

\put(100,0){\circle*{2}}
\put(110,0){\circle*{2}}
\put(100,10){\circle*{2}}
\put(110,10){\circle*{2}}
\put(105.4,5.4){\circle*{2}}
\put(115.4,5.4){\circle*{2}}
\put(105.4,15.4){\circle{2}}
\put(115.4,15.4){\circle{2}}
\put(100,1){\line(0,1){8}}
\put(101,0){\line(1,0){8}}
\put(110,1){\line(0,1){8}}
\put(101,10){\line(1,0){8}}

\put(105.4,6.4){\line(0,1){8}}
\put(106.4,5.4){\line(1,0){8}}
\put(115.4,6.4){\line(0,1){8}}
\put(106.4,15.4){\line(1,0){8}}

\put(100.7,0.7){\line(1,1){4}}
\put(110.7,0.7){\line(1,1){4}}
\put(100.7,10.7){\line(1,1){4}}
\put(110.7,10.7){\line(1,1){4}}

\put(120,0){\circle*{2}}
\put(130,0){\circle*{2}}
\put(120,10){\circle*{2}}
\put(130,10){\circle*{2}}
\put(125.4,5.4){\circle*{2}}
\put(135.4,5.4){\circle*{2}}
\put(125.4,15.4){\circle*{2}}
\put(135.4,15.4){\circle{2}}
\put(120,1){\line(0,1){8}}
\put(121,0){\line(1,0){8}}
\put(130,1){\line(0,1){8}}
\put(121,10){\line(1,0){8}}

\put(125.4,6.4){\line(0,1){8}}
\put(126.4,5.4){\line(1,0){8}}
\put(135.4,6.4){\line(0,1){8}}
\put(126.4,15.4){\line(1,0){8}}

\put(120.7,0.7){\line(1,1){4}}
\put(130.7,0.7){\line(1,1){4}}
\put(120.7,10.7){\line(1,1){4}}
\put(130.7,10.7){\line(1,1){4}}

\put(0,50){\circle{2}}
\put(10,50){\circle{2}}
\put(0,60){\circle{2}}
\put(10,60){\circle{2}}
\put(5.4,55.4){\circle{2}}
\put(15.4,55.4){\circle{2}}
\put(5.4,65.4){\circle{2}}
\put(15.4,65.4){\circle{2}}
\put(0,51){\line(0,1){8}}
\put(1,50){\line(1,0){8}}
\put(10,51){\line(0,1){8}}
\put(1,60){\line(1,0){8}}

\put(5.4,56.4){\line(0,1){8}}
\put(6.4,55.4){\line(1,0){8}}
\put(15.4,56.4){\line(0,1){8}}
\put(6.4,65.4){\line(1,0){8}}

\put(0.7,50.7){\line(1,1){4}}
\put(10.7,50.7){\line(1,1){4}}
\put(0.7,60.7){\line(1,1){4}}
\put(10.7,60.7){\line(1,1){4}}

\put(20,50){\circle*{2}}
\put(30,50){\circle{2}}
\put(20,60){\circle{2}}
\put(30,60){\circle{2}}
\put(25.4,55.4){\circle{2}}
\put(35.4,55.4){\circle{2}}
\put(25.4,65.4){\circle{2}}
\put(35.4,65.4){\circle{2}}
\put(20,51){\line(0,1){8}}
\put(21,50){\line(1,0){8}}
\put(30,51){\line(0,1){8}}
\put(21,60){\line(1,0){8}}

\put(25.4,56.4){\line(0,1){8}}
\put(26.4,55.4){\line(1,0){8}}
\put(35.4,56.4){\line(0,1){8}}
\put(26.4,65.4){\line(1,0){8}}

\put(20.7,50.7){\line(1,1){4}}
\put(30.7,50.7){\line(1,1){4}}
\put(20.7,60.7){\line(1,1){4}}
\put(30.7,60.7){\line(1,1){4}}

\put(40,50){\circle*{2}}
\put(50,50){\circle*{2}}
\put(40,60){\circle{2}}
\put(50,60){\circle{2}}
\put(45.4,55.4){\circle{2}}
\put(55.4,55.4){\circle{2}}
\put(45.4,65.4){\circle{2}}
\put(55.4,65.4){\circle{2}}
\put(40,51){\line(0,1){8}}
\put(41,50){\line(1,0){8}}
\put(50,51){\line(0,1){8}}
\put(41,60){\line(1,0){8}}

\put(45.4,56.4){\line(0,1){8}}
\put(46.4,55.4){\line(1,0){8}}
\put(55.4,56.4){\line(0,1){8}}
\put(46.4,65.4){\line(1,0){8}}

\put(40.7,50.7){\line(1,1){4}}
\put(50.7,50.7){\line(1,1){4}}
\put(40.7,60.7){\line(1,1){4}}
\put(50.7,60.7){\line(1,1){4}}

\put(60,50){\circle*{2}}
\put(70,50){\circle{2}}
\put(60,60){\circle{2}}
\put(70,60){\circle*{2}}
\put(65.4,55.4){\circle{2}}
\put(75.4,55.4){\circle{2}}
\put(65.4,65.4){\circle{2}}
\put(75.4,65.4){\circle{2}}
\put(60,51){\line(0,1){8}}
\put(61,50){\line(1,0){8}}
\put(70,51){\line(0,1){8}}
\put(61,60){\line(1,0){8}}

\put(65.4,56.4){\line(0,1){8}}
\put(66.4,55.4){\line(1,0){8}}
\put(75.4,56.4){\line(0,1){8}}
\put(66.4,65.4){\line(1,0){8}}

\put(60.7,50.7){\line(1,1){4}}
\put(70.7,50.7){\line(1,1){4}}
\put(60.7,60.7){\line(1,1){4}}
\put(70.7,60.7){\line(1,1){4}}

\put(80,50){\circle*{2}}
\put(90,50){\circle{2}}
\put(80,60){\circle{2}}
\put(90,60){\circle{2}}
\put(85.4,55.4){\circle{2}}
\put(95.4,55.4){\circle{2}}
\put(85.4,65.4){\circle{2}}
\put(95.4,65.4){\circle*{2}}
\put(80,51){\line(0,1){8}}
\put(81,50){\line(1,0){8}}
\put(90,51){\line(0,1){8}}
\put(81,60){\line(1,0){8}}

\put(85.4,56.4){\line(0,1){8}}
\put(86.4,55.4){\line(1,0){8}}
\put(95.4,56.4){\line(0,1){8}}
\put(86.4,65.4){\line(1,0){8}}

\put(80.7,50.7){\line(1,1){4}}
\put(90.7,50.7){\line(1,1){4}}
\put(80.7,60.7){\line(1,1){4}}
\put(90.7,60.7){\line(1,1){4}}

\put(100,50){\circle*{2}}
\put(110,50){\circle*{2}}
\put(100,60){\circle*{2}}
\put(110,60){\circle{2}}
\put(105.4,55.4){\circle{2}}
\put(115.4,55.4){\circle{2}}
\put(105.4,65.4){\circle{2}}
\put(115.4,65.4){\circle{2}}
\put(100,51){\line(0,1){8}}
\put(101,50){\line(1,0){8}}
\put(110,51){\line(0,1){8}}
\put(101,60){\line(1,0){8}}

\put(105.4,56.4){\line(0,1){8}}
\put(106.4,55.4){\line(1,0){8}}
\put(115.4,56.4){\line(0,1){8}}
\put(106.4,65.4){\line(1,0){8}}

\put(100.7,50.7){\line(1,1){4}}
\put(110.7,50.7){\line(1,1){4}}
\put(100.7,60.7){\line(1,1){4}}
\put(110.7,60.7){\line(1,1){4}}

\put(120,50){\circle*{2}}
\put(130,50){\circle*{2}}
\put(120,60){\circle{2}}
\put(130,60){\circle{2}}
\put(125.4,55.4){\circle{2}}
\put(135.4,55.4){\circle{2}}
\put(125.4,65.4){\circle{2}}
\put(135.4,65.4){\circle*{2}}
\put(120,51){\line(0,1){8}}
\put(121,50){\line(1,0){8}}
\put(130,51){\line(0,1){8}}
\put(121,60){\line(1,0){8}}

\put(125.4,56.4){\line(0,1){8}}
\put(126.4,55.4){\line(1,0){8}}
\put(135.4,56.4){\line(0,1){8}}
\put(126.4,65.4){\line(1,0){8}}

\put(120.7,50.7){\line(1,1){4}}
\put(130.7,50.7){\line(1,1){4}}
\put(120.7,60.7){\line(1,1){4}}
\put(130.7,60.7){\line(1,1){4}}
\end{picture}
\end{equation*}
\caption{Representatives for the motion equivalence classes $\eta_j$, $j=1,\dots,21$ shown in white.}
\label{eta}
\end{figure}
Since $Z$ is isotropic, we may as well let the weights be motion independent, i.e.\ for all configurations $(B_l,W_l)$ with $W_l\in \eta_j$ we choose the weight $w_l^{(q)}$ equal to some weight $\tilde{w}_j^{(q)}$ depending only on $j$, see \cite{am} for a justification. By \eqref{meanHM} and Proposition \ref{hBW} we must set $\tilde{w}_{22}^{(q)}=0$ for all $q<d$ in order to obtain convergent algorithms. Thus \eqref{loces} simplifies to
\begin{align*}
\hat{V}_q(Z\cap A){}&= a^{q-d} \sum_{j = 1}^{21} \tilde{w}_j^{(q)}  \sum_{l:W_l\in \eta_j }\frac{N_l(Z\cap A\cap a\La)}{N(A)}.
\end{align*}

Let $D\in\R^{21\times 21}$ be the diagonal matrix with $D_{ii}=|\eta_i|$, $1\leq i\leq 21$ (see Table \ref{Ver}) and let $(B_{l_j},W_{l_j})$ be a $2\times 2\times 2$ configuration belonging to the equivalence class $\eta_j$. Moreover, let $w^{(q)}=(\tilde{w}_1^{(q)},\ldots,\tilde{w}_{21}^{(q)})$ and
$c_q=(c_q(B_{l_1},W_{l_1}),\ldots,c_q(B_{l_{21}},W_{l_{21}}))$, $1\leq q\leq 3$.
Then, Theorem \ref{mean1stIsotropic} implies 
\begin{align}\label{EqV2Iso3D}
E \hat{V}_{2}(Z\cap A)&= \gamma e^{-\gamma EV_3(K)}w^{(2)}Dc_1^\top EV_{2}(K)+O(a)
\end{align}
and under condition \eqref{cond3}
\begin{align}\label{EqV1Iso3D}
E\hat{V}_{1}(Z\cap A)&= e^{-\gamma EV_3(K)}w^{(1)} D\big(\gamma c_2^\top EV_{1}(K)\nonumber\\
&\quad +\frac{\gamma^2}{2}c_3^\top (E V_{2}(K))^2\big)+O(a).
\end{align}
Since the constants $c_k(B_l,W_l)$ are independent of the grain distribution and a Boolean model with balls as grains is a special case of an isotropic Boolean model, 
it is enough to consider this situation in order to determine the constants $c_{k}(B_l,W_l)$.  We study this choice in detail in the next section.
Furthermore if the typical grain is a ball with random radius $r$ equation \eqref{EstMatrixEq} which is shown in the next section implies  
\[
E \hat{V}_q(Z\cap A)=a^{q-3}w^{(q)}DQv(a)^\top+O(a^{q+1}),
\]
where the matrix $Q\in\R^{21\times 8}$ is defined in \eqref{DefnQ} (see also Table \ref{Qer}) and $v(a)\in \R^{21}$ in \eqref{Defnva}.
Comparing the summand independent of $a$ with \eqref{EqV2Iso3D} respectively \eqref{EqV1Iso3D} we obtain 
\[
-Q^j_4\gamma 2E r+Q^j_5\frac{1}{2}\gamma^2 \pi^2 (E r^2)^2=\gamma c_2^j4 Er+\frac{1}{2}\gamma^2 c_3^j 4\pi^2 (E r^2)^2
\]
and
\[
-Q^j_3 \gamma Er^2\pi= \gamma c_1^j 2\pi Er^2.
\]
Thus $c_1(B_{l_j},W_{l_j}) = -\frac{1}{2}Q^j_3$, $c_2(B_{l_j},W_{l_j})=-\frac{1}{2}Q^j_4$ and $c_3(B_{l_j},W_{l_j})=\frac{1}{4}Q^j_5$. For $k=1,3$, $c_k(B_l,W_l)$ were also computed directly in \cite{am2}. 

Inserting this in Theorem \ref{mean1stIsotropic} we obtain the following corollary. 
\begin{cor}\label{DigAlgIso3d}
Let $Z$ be a stationary, isotropic Boolean model in $\R^{3}$.
If Condition \ref{cond1} is satisfied and $w^{(2)}_1=w^{(2)}_{22}=0$, then
\[
E \hat{V}_2(Z\cap A)=-\frac{1}{2}w^{(2)}D\gamma e^{-\gamma E V_3(K)}Q_3 EV_2(K)+O(a).
\]
If Condition \ref{cond2} is satisfied, $w^{(1)}_1=w^{(1)}_{22}=0$ and $w^{(1)}DQ_3=0$, then
\[
\hat{V}_1(Z\cap A)=w^{(1)}De^{-\gamma EV_3(K)}\left[-\frac{\gamma}{2}Q_4EV_1(K)+\frac{\gamma^2}{8}Q_5(EV_2(K))^2\right]+O(a).
\]
\end{cor}
Now we obtain conditions on optimal weights of local algorithms for the estimation of $\overline{V}_2(Z)$ and $\overline{V}_1(Z)$. 
\begin{thm}
Let $Z$ be a stationary, isotropic Boolean model in $\R^{3}$.
Let Condition \ref{cond1} be satisfied. Then, $\hat{V}_2(Z\cap A)$ is an asymptotically unbiased estimator of $\overline{V}_2(Z)$ if 
\[w^{(2)}D Q_3=-2 \text{ and } w^{(2)}_1=w^{(2)}_{22}=0.\]
Let Condition \ref{cond2} be satisfied. Then $\hat{V}_1(Z\cap A)$ is an asymptotically unbiased estimator of $\overline{V}_1(Z)$ if
\[w^{(1)}D Q_4=-2, \quad w^{(1)}D Q_5=-\pi\]
and 
\[
w^{(1)}DQ_3=w^{(1)}_1=w^{(1)}_{22}=0.
\]
This is satisfied by the weights in Table \ref{optimal}.
\end{thm}\label{UnbiasIso3d}
\begin{proof}
The assertion follows from a comparison of Corollary \ref{DigAlgIso3d} with the Miles formulas \eqref{Miles}. The weights in Table \ref{optimal} fulfill the asserted condition since they fulfill \eqref{ligning} and $Q_1,\ldots,Q_8$ are the columns of the matrix $Q$.
\end{proof}

The weights $w^{(1)}$ from Table \ref{optimal} are also optimal based on the results of \cite{am2} for the design based setting where an $r$-regular set is observed on a randomly translated and rotated lattice.
This follows since $w^{(1)}DQ_3=0$ and $w^{(1)}_1=0$ imply the first condition on the weights in \cite[Cor. 5.1 (9)]{am2} and hence the convergence of the estimator, and $w^{(1)}DQ_4=-2$ implies the second condition of \cite[Cor. 5.1 (9)]{am2} and hence that the estimator in the same theorem is unbiased. 
Thus, the weights in Table \ref{optimal} are an optimal choice for isotropic Boolean models in $\R^{3}$ with compact convex grains satisfying Condition \ref{cond2}. But in particular, they are also optimal based on the results of \cite{am2} for the design based setting where an $r$-regular set is observed on a randomly translated and rotated lattice.

\section{Optimal algorithms for $3D$ Boolean models with balls as grains}\label{Balls}

We now consider a stationary Boolean model whose grains are a.s.\ random balls of radius $r\geq \eps$ for some fixed $\eps>0$. The choice of balls as grains is also the situation studied in \cite{nagel}. 
In this situation we can show a refined third order version of Lemma \ref{convbound} with $K$ replaced by a ball using intrinsic power volumes. 
This third order expansion will allow us to strengthen the previous results. 

\subsection{Intrinsic power volumes}
The intrinsic power volumes $V_j^{(m)}$ are positive and $m$-homogeneous functionals on finite subsets of $\R^d$ introduced in \cite{markus}.
The key ingredient for the refinement of Lemma \ref{convbound} is
the following result of \cite[Corollary 6]{markus}: 
\begin{align}
\MoveEqLeft V_3(\conv(F) \oplus rB^3)-V_3(F \oplus rB^3)\nonumber\\
&= \pi V_1^{(3)}(F) + 2\sum_{n=1}^{\infty} \frac{(2n-3)!!}{2n!!}V_2^{(2n+2)}(F)r^{-(2n-1)}\label{intrinsicPowerVol}
\end{align}
which holds whenever $F$ is a finite set satisfying Condition (A) of that paper and $r$ is sufficiently large. 
Let  $F\subseteq C^2_{0,0}$ be nonempty. Then $F$ is the vertex set of $\conv(F)$ and no three points in $F$ form a triangle with a strictly obtuse angle. Thus Condition (A) of \cite{markus} is satisfied for $F$ as explained in this paper.
Moreover, $V_1^{(3)}$ is given by the following formula \cite[Equation (17)]{markus}:
\begin{equation*}
V_1^{(3)}(F)= \frac{1}{12}\sum_{H\in \mathcal{F}_1(\conv (F))}\gamma(\conv(F),H)V_1(H)^3.
\end{equation*}
Here $\mathcal{F}_1(\conv (F))$ is the set of 1-faces in $\conv (F)$ and $\gamma(\conv(F),H)$ is the exterior angle, see \cite[Equation (3.2)]{markus}.

Now, for sufficiently large $\frac{r}{a}$ an application of \eqref{intrinsicPowerVol} implies
\begin{align*}
\MoveEqLeft V_3(a\conv(F) \oplus rB^3)-V_3(aF \oplus rB^3){}\\
&= \pi V_1^{(3)}(aF) + 2\sum_{n=1}^{\infty} \frac{(2n-3)!!}{2n!!}V_2^{(2n+2)}(aF)r^{-(2n-1)}\\
&= a^3 \pi V_1^{(3)}(F) + 2\sum_{n=1}^{\infty} \frac{(2n-3)!!}{2n!!}V_2^{(2n+2)}(F)a^{2n+2}r^{-(2n-1)}.
\end{align*}
Since $\frac{r}{a}\geq \frac{\eps}{a}$ a.s.\ and all coefficients are positive, 
\begin{align} \label{powervol}
\MoveEqLeft EV_3(a\conv(F) \oplus rB^3)-EV_3(aF \oplus rB^3)\\
&= a^3 \pi V_1^{(3)}(F) + 2\sum_{n=1}^{\infty} \frac{(2n-3)!!}{2n!!}V_2^{(2n+2)}(F)a^{2n+2}E(r^{-(2n-1)})\nonumber
\end{align}
for sufficiently small $a$. 
The formulas for the intrinsic power volumes $V_2^{(2n+2)}$ are rather involved, so the above formula is not suitable for general computations. 
However, we obtain
\begin{align*} 
EV_3(a\conv(F) \oplus rB^3)-EV_3(aF \oplus rB^3)-a^3\pi V_1^{(3)}(F)\in O(a^4).
\end{align*}

Now \eqref{exp}, \eqref{powervol} and the Steiner formula \eqref{steiner} yield 
\begin{align*}
\MoveEqLeft P(aF\subseteq \R^3 \backslash Z)=\exp\Big(-\gamma \Big[\frac{4}{3}\pi E r^3+a V_1(\conv F) E r^2\pi \\
&+a^2 V_2(\conv F)2E r+a^3\left(V_3(\conv F)-\pi V_1^{(3)}(F)\right)+O(a^4)\Big]\Big).
\end{align*}
A development of the exponential function implies the third order expansion
\begin{align*}
P(aF\subseteq \R^3 \backslash Z)&=\exp(-\gamma \frac{4}{3}\pi E r^3)
\Big(1-a\gamma\pi E r^2V_1(\conv F)+a^2\frac{\gamma^2\pi^2 (E r^2)^2}{2}V_1(\conv F)^2\\
&\quad-a^3\frac{\gamma^3\pi^3 (E r^2)^3}{6}V_1(\conv F)^3+a^3 2\gamma^2\pi E r Er^2 V_1(\conv F) V_2(\conv F)\\
&\quad-a^2\gamma 2E r V_2(\conv F)-a^3\gamma V_3(\conv F) +a^3 \gamma \pi V_1^{(3)}(F)\Big).
\end{align*}

Now define
\begin{align}
 v(a)=e^{-\gamma \frac{4}{3}\pi E r^3} &\Big(e^{\gamma \frac{4}{3}\pi E r^3},\, 1,\, -a \gamma Er^2 \pi ,\, -a^2 \gamma 2 Er ,\, a^2\frac{\gamma^2\pi^2 (Er^2)^2}{2},\nonumber\\
&\;\;  -a^3\gamma ,\,  a^3 \gamma^2 2\pi E r Er^2,-a^3 \frac{\gamma^3 }{6} \pi^3 (E r^2)^3 \Big)\label{Defnva}.
\end{align}
For $1\leq i\leq 21$ and $S \in \eta_i$ we define $p^i=P(aS \subseteq \R^3 \backslash Z)$. Let $P^i =(P_1^i,\dots,P_8^i)$ be the vector 
\begin{align*}
P^i ={}& (0 ,1, V_1(\conv S), V_2(\conv S), V_1(\conv S)^2, \\
& V_3(\conv S )-\pi V_1^{(3)}(S),  V_1(\conv S )V_2(\conv S ), V_1(\conv S )^3).
\end{align*}
 Then
\begin{equation}\label{pjequation}
p^i = P^i  v(a)^T +O(a^4).
\end{equation}

The values needed to compute $P^i$ for $i\neq 21$ are given in Table \ref{Ver} in the appendix, see also \cite[Table 4]{nagel} for the first three columns.
For $W_{l_i}\in \eta_i$, the configuration $(B_{l_i},W_{l_i})$ satisfies by \eqref{EI1} and \eqref{pjequation} the relation
\begin{align}
P(a B_{l_i}\subseteq Z,a W_{l_i} \subseteq \R^3 \backslash Z) {}&=\sum_{j=1}^{21} \sum_{S\subseteq B_{l_i}} (-1)^{|S|} p^j\mathds{1}_{\{W_{l_i}\cup S \in \eta_j\}} \nonumber\\
&=\bigg(\sum_{j=1}^{21} \sum_{S\subseteq B_{l_i}} (-1)^{|S|} P^j\mathds{1}_{\{W_{l_i}\cup S \in \eta_j\}}\bigg)\cdot v(a)^T+O(a^4)\nonumber\\
&= Q^i \cdot v(a)^T+O(a^4) \label{HitAndMiss1}
\end{align}
where 
\begin{equation}\label{DefnQ}
 Q^i  = (Q_1^i,\dots,Q_8^i) =\sum_{j=1}^{21} \left(\sum_{S\subseteq B_{l_i}} (-1)^{|S|}\mathds{1}_{ \{W_{l_i}\cup S \in \eta_j\}}\right)P^j.
\end{equation}

Writing $Q=\begin{pmatrix}Q^1\\ \vdots\\ \;\;Q^{21} \end{pmatrix}$ and 
 $P=\begin{pmatrix}P^1\\ \vdots\\ \;\;P^{21} \end{pmatrix}$, we thus get a matrix $M$ such that
\begin{equation*}
Q= M P
\end{equation*}
where the entries of $M$ are given by
\begin{equation*}
(M)_{ij}= \sum_{S\subseteq B_{l_i}} (-1)^{|S|} \mathds{1}_{\{W_{l_i}\cup S \in \eta_j\}}.
\end{equation*}
The matrix $M$ is shown in Table \ref{M} in the appendix. Clearly, $Q_{1}^j=\mathds{1}_{\{j=1\}}-\mathds{1}_{\{j=22\}}$ and $Q_{2}^j=\mathds{1}_{\{j=22\}}$. The values of $(Q_3^j,\dots,Q^j_6)$ are given in the appendix Table \ref{Qer}.

Let $w^{(q)}=(\tilde{w}_1^{(q)},\dots,\tilde{w}_{21}^{(q)})$.
By \eqref{meanHM} and since the configurations of one motion equivalence class all have the same weight, the mean of a local algorithm is thus given by
\begin{align*}
E\hat{V}_q(Z\cap A) &=a^{q-3} \sum\limits_{j=1}^{21} \tilde{w}_j^{(q)} |\{l:W_l\in\eta_j\}| P(aB_{l_j}\subseteq Z, aW_{l_j}\subseteq \R^d\setminus Z).
\end{align*}
Now it follows from \eqref{HitAndMiss1} that
\begin{equation}\label{EstMatrixEq}
E\hat{V}_q(Z\cap A) =a^{q-3} w^{(q)} D Q v(a)^\top+O(a^{q+1}).
\end{equation}

Note that using $V_0(rB^3)=1$, $V_1(rB^3)=4r$, $V_2(rB^3)=2\pi r^2$ and $V_3(rB^3)=\frac{4}{3}\pi r^3$, the Miles formulas \eqref{Miles} can be written as 
\begin{equation*}
\altoverline{V}_q(Z)=a^{q-3}v(a) b_q^T, 
\end{equation*}
where $0\leq q\leq 3$ and
\begin{align}\label{ber}
b_3 &= (1,-1,0,0,0,0,0,0)\\ \nonumber
b_2 &= (0,0,-2,0,0,0,0,0)\\ \nonumber
b_1 &= (0,0,0,-2,-{\pi},0,0,0)\\ \nonumber
b_0 &= (0,0,0,0,0,-1, -2,-{\pi}).
\end{align}

In particular, the best possible local algorithm for $\altoverline{V}_q(Z)$ based on the computations of this section would be one that satisfies
\begin{equation}\label{ligning}
w^{(q)} D \,Q = b_q.
\end{equation}
This can be used to check how well suited an established algorithm is for Boolean models, as we shall see in the next section.


\subsection{Application to algorithms}\label{algor}
In \cite[Table 1]{nagel} a set of weights is suggested based on a discretisation of the Crofton formula, using an approximation of $Z$ by 4 different adjacency systems. For each algorithm, we apply the above to compute the left hand side of \eqref{ligning}. The outcomes are shown as row vectors in  Table \ref{meannagel}. These should be compared to the optimal values $b_q$.
\begin{table}
\footnotesize
\centering
\begin{tabular}{ccc@{\,}c@{\,}c@{\,}c@{\,}c@{\,}c@{\,}c@{\,}c}
\hline
$\altoverline{V}_i$&Adjacency system & &&&&$w^{(q)}D\, Q$&&&\\
\hline
$\altoverline{V}_2$&All& 0&0&-2& 0&-2.7798&0.5253&-0.0015&-4.0161\\
$\altoverline{V}_1$&All& 0&0&0&-2&-3.6096&0&-3.9733&-11.7843\\
$\altoverline{V}_0$&$(\mathbb{F}_6,\mathbb{F}_{26})$&0&0&0&0&-0.0131&-1&-2.1895&-3.6284\\
&& &&&& (-0.0130)&  &(-2,19) & (3,62)\\
$\altoverline{V}_0$&$(\mathbb{F}_{14.1},\mathbb{F}_{14.1})$&0&0&0&0&-0.0399&-1&-2.6286&-4.9038\\
&& &&&& (-0.0399)&  &(-0,42) & (-4.90)\\
$\altoverline{V}_0$&$(\mathbb{F}_{14.2},\mathbb{F}_{14.2})$&0&0&0&0&-0.0460&-1&-2.6461&-4.9786\\
&& &&&& (-0.105)&  &(-0,44) & (-5.34)\\
$\altoverline{V}_0$&$(\mathbb{F}_{26},\mathbb{F}_{6})$&0&0&0&0&0&-1&-3&-6\\
&& &&&& (0)&  &(-3) & (-6)\\
\hline
\end{tabular}
\caption{Mean of the algorithms suggested in \cite{nagel}. This should be compared to the true values \eqref{ber}. The values computed in \cite{nagel} are shown in parenthesis.}
\label{meannagel}
\end{table}

The computations of the asymptotic bias are also made in \cite{nagel} up to second order. The third order term is approximated by leaving out the unknown contribution from $V_1^{(3)}$. Surprisingly, we see that these terms do not contribute. 
The $\altoverline{V}_2$ estimator is asymptotically unbiased, but there is a bias of order $a$. The estimator for $\altoverline{V}_1$ is biased and the estimators for $\altoverline{V}_0$ do not even converge when $a\to 0$ except for one of them, which instead has a large bias. This was already observed in \cite{nagel}.

We remark here that the constants in Table \ref{meannagel} differ from those in \cite[Table 4]{nagel}, which are again different from those computed in \cite[Table 1]{oh}. While most of the numbers agree for three of the algorithms for $\altoverline{V}_0$, they are far off for one of them. We have not been able to find an explanation for this.

We suggest instead to estimate $\altoverline{V}_q$ by means of an algorithm that satisfies \eqref{ligning} since this will not only be asymptotically unbiased but in finite resolution the bias will only be of order $O(a^{3-q+1})$. A set of weights satisfying \eqref{ligning} is given by Table \ref{optimal}. Of course, adding any solution to the homogeneous system $w^{(q)}D\,Q=0$ yields another set of weights that may be just as good asymptotically. 
\section*{Appendix}
In this appendix, we collect some tables of values computed in the paper.

\begin{table}[H]\footnotesize
\centering
\begin{tabular}{|c@{}c@{}c@{}c@{}c@{}c@{}c@{}c@{}c@{}c@{}c@{}c@{}c@{}c@{}c@{}c@{}c@{}c@{}c@{}c@{}c|}
1& 0& 0& 0& 0 & 0&0&0&0&0&0&0&0&0&0&0&0&0&0&0&0 \\ 
-1& 1& 0 & 0& 0 & 0&0&0&0&0&0&0&0&0&0&0&0&0&0&0&0 \\
1& -2& 1& 0& 0 & 0&0&0&0&0&0&0&0&0&0&0&0&0&0&0&0 \\
 1& -2& 0&1 & 0 & 0&0&0&0&0&0&0&0&0&0&0&0&0&0&0&0 \\ 
 1&  -2& 0 &0&1 &0&0&0&0&0&0&0&0&0&0&0&0&0&0&0&0 \\
 -1&  3& -2&-1&0 & 1 &0&0&0&0&0&0&0&0&0&0&0&0&0&0&0 \\
- 1&  3& -1 & -1& -1 & 0&1&0&0&0&0&0&0&0&0&0&0&0&0&0&0 \\
 -1&  3&  0& -3& 0 & 0&0&1&0&0&0&0&0&0&0&0&0&0&0&0&0 \\
 1&  -4&  4& 2& 0 & -4&0&0&1&0&0&0&0&0&0&0&0&0&0&0&0 \\
 1&  -4&  3& 3& 0 & -3&0&-1&0&1&0&0&0&0&0&0&0&0&0&0&0 \\
 1&  -4&  3& 2& 1 & -2&-2&0&0&0&1&0&0&0&0&0&0&0&0&0&0 \\
 1&  -4&  2& 2& 2 & 0&-4&0&0&0&0&1&0&0&0&0&0&0&0&0&0 \\
 1&  -4&  0& 6& 0 & 0&0&-4&0&0&0&0&1&0&0&0&0&0&0&0&0 \\
 1&  -4& 2& 3& 1 & -1&-2&-1&0&0&0&0&0&1&0&0&0&0&0&0&0 \\
 -1&  5&  -3& -6& -1 & 3&3&4&0&-1&0&0&-1&-3&1&0&0&0&0&0&0 \\
 -1&  5& -4 & -4& -2 & 3&6&1&0&0&-2&-1&0&-2&0&1&0&0&0&0&0 \\
 -1&  5&  -5& -4& -1 & 6&3&1&-1&-1&-2&0&0&-1&0&0&1&0&0&0&0 \\
 1&  -6& 6 & 6& 3 & -6&-12&-2&0&0&6&3&0&6&0&-6&0&1&0&0&0 \\
 1&  -6& 6 & 7& 2 & -8&-8&-4&1&2&4&1&1&6&-2&-2&-2&0&1&0&0 \\
 1&  -6& 7 & 6& 2 & -10&-8&-2&2&2&6&1&0&4&0&-2&-4&0&0&1&0 \\
 -1&  7&  -9& -9& -3 & 15&15&5&-3&-4&-12&-3&-1&-12&3&9&9&-1&-3&-3&1 \\
\end{tabular}
\caption{The matrix $M$.}
\label{M}
\end{table}

\begin{table}[H]\footnotesize
\centering
\begin{tabular}{c@{}c@{}c@{}c@{}c@{}c}
\hline
$\eta$&$V_3(F)$& $V_2(F)$& $V_1(F)$& $24 V_1^{(3)}(F)$&$D_{jj}$\\ 
\hline 
$\eta_1$&1& 3& 3 &3&1 \\
$\eta_2$&$\frac{5}{6}$& $\frac{9}{4}+\frac{ \sqrt{3}}{4}$& $\frac{9}{4}+3{\sqrt{2}}\xi$& $\frac{9}{4}+6{\sqrt{2}}\xi$&8\\
$\eta_3$&$\frac{1}{2}$& $\frac{3}{2}+\frac{\sqrt{2}}{2}$&$ 2+\frac{\sqrt{2}}{2}$& $2+\sqrt{2}$&12\\ 
$\eta_4$& $\frac{2}{3}$& $\frac{3}{2}+\frac{\sqrt{3}}{2}$& $\frac{3}{2}+6{\sqrt{2}}\xi $&$\frac{3}{2}+12{\sqrt{2}}\xi $ &12\\
$\eta_5$& $\frac{2}{3}$& $\frac{3}{2}+\frac{\sqrt{3}}{2}$& $\frac{3}{2}+6{\sqrt{2}}\xi $&$\frac{3}{2}+12{\sqrt{2}}\xi $ &4\\
$\eta_6$& $\frac{1}{3}$& $ 1+\frac{\sqrt{2}}{2}$& $\frac{3}{2}+\frac{\sqrt{2}}{2}+\frac{\sqrt{3}}{6}$&$\frac{3}{2}+{\sqrt{2}}+\frac{\sqrt{3}}{2}$&24 \\
$\eta_7$& $\frac{1}{3}$& $\frac{3}{4}+\frac{\sqrt{2}}{2}+\frac{\sqrt{3}}{4}$& $ \frac{5}{4}+\frac{\sqrt{2}}{2}+3{\sqrt{2}}\xi $&$ \frac{5}{4}+{\sqrt{2}}+6{\sqrt{2}}\xi $ &24\\
$\eta_8$& $\frac{1}{2}$& $ \frac{3}{4}+\frac{3\sqrt{3}}{4}$& $ \frac{3}{4}+9{\sqrt{2}}\xi$& $ \frac{3}{4}+18{\sqrt{2}}\xi$&8\\
$\eta_9$& 0&  1&  2& 2&6\\
$\eta_{10}$& $\frac{1}{6}$&  $\frac{3}{4}+\frac{\sqrt{3}}{4}$& $ \frac{3}{4}+\frac{3\sqrt{2}}{2}-3\sqrt{2}\xi $&$ \frac{3}{4}+{3\sqrt{2}}-6\sqrt{2}\xi $ &8\\
$\eta_{11}$& $\frac{1}{6}$& $ \frac{1}{2}+\frac{\sqrt{2}}{2}$& $ 1+\frac{\sqrt{2}}{2}+\frac{\sqrt{3}}{3}$& $ 1+{\sqrt{2}}+{\sqrt{3}}$ &24\\
$\eta_{12}$& 0& $ {\sqrt{2}}$&  $1+\sqrt{2}$ & $1+2\sqrt{2}$  &6\\
$\eta_{13}$& $\frac{1}{3}$&  $\sqrt{3}$& $ 12\sqrt{2}\xi $  & $ 24\sqrt{2}\xi $ &2\\
$\eta_{14}$& $\frac{1}{6}$&  $\frac{1}{4}+\frac{\sqrt{2}}{2}+\frac{\sqrt{3}}{4}$ & $ \frac{3}{4}+\frac{\sqrt{2}}{2}+\frac{\sqrt{3}}{6}+3\sqrt{2}\xi$& $ \frac{3}{4}+{\sqrt{2}}+\frac{\sqrt{3}}{2}+6\sqrt{2}\xi$& 24\\
$\eta_{15}$& 0&  $\frac{\sqrt{3}}{2}$& $\frac{3\sqrt{2}}{2} $& ${3\sqrt{2}} $ &8\\
$\eta_{16}$& 0& $\frac{\sqrt{2}}{2}$& $ \frac{1}{2}+\frac{\sqrt{2}}{2}+\frac{\sqrt{3}}{2}$&$ \frac{1}{2}+{\sqrt{2}}+\frac{3\sqrt{3}}{2}$&24 \\
$\eta_{17}$& 0& $ \frac{1}{2}$& $1+\frac{\sqrt{2}}{2} $&$1+{\sqrt{2}} $ &24 \\
$\eta_{18}$& 0&  0& $\sqrt{3}$&$3\sqrt{3}$  &4\\
$\eta_{19}$& 0&  0& ${\sqrt{2}}$&$2{\sqrt{2}}$&12  \\
$\eta_{20}$& 0&  0&  1& 1&12\\
$\eta_{21}$& 0&  0&  0& 0&8\\
\hline 
\end{tabular}
\caption{List of $V_q(F)$ and $V_1^{(3)}(F)$ for $F\in \eta_j$. Here $\xi= \frac{\arctan(\sqrt{2})}{2\pi }$.}
\label{Ver}
\end{table}

\begin{table}[H]\footnotesize
\centering
\begin{tabular}{ccccc}
\hline
$j$&$ Q^j_3$& $Q^j_4$& $Q^j_5$& $ Q^j_6$\\ 
\hline 
$1$& 3& 3& 9 & $1 - \frac{\pi}{8}$\\
$2$&$-\frac{3}{4}+3\sqrt{2}\xi$& $\frac{ \sqrt{3}-3}{4}$& -0.6186 &$ - \frac{4+\pi(-\frac{3}{4}+6\sqrt{2}\xi)}{24}$  \\
$3$&$\frac{1}{2}-6\sqrt{2}\xi+\frac{\sqrt{2}}{2}$& $\frac{\sqrt{2}-\sqrt{3}}{2}$& -0.4344  &$ -\frac{4+\pi(\frac{1}{2}-12\sqrt{2}\xi+\sqrt{2})}{24}$ \\ 
$4$& $0$& $0$& 0.02203&0  \\
$5$& $0$& $0$&0.02203 & 0 \\
$6$& $-\frac{1}{4}-\frac{\sqrt{2}}{2}+3\sqrt{2}\xi+\frac{\sqrt{3}}{6}$& $ \frac{1-2\sqrt{2}+\sqrt{3}}{4}$&-0.06855 &$ \frac{4-\pi(-\frac{1}{4}-\sqrt{2}+6\sqrt{2}\xi+\frac{\sqrt{3}}{2})}{24}$   \\
$7$& $0$& $0$&0.0174 & 0 \\
$8$& $0$& $ 0$& 0 & 0 \\
$9$& $1-\frac{2\sqrt{3}}{3}$&  0& -0.5580& $-\frac{1}{3} - \frac{\pi}{24}(1-2\sqrt{3})$  \\
${10}$& $\frac{3\sqrt{2}}{2}-6\sqrt{2}\xi-\frac{\sqrt{3}}{2}$& 0 & - 0.1267 &  $-\frac{4+\pi({3}\sqrt{2}-12\sqrt{2}\xi-\frac{3\sqrt{3}}{2})}{24}$  \\
${11}$& $0$& $ 0$& 0.03245& 0\\
${12}$& 0& $ 0$& 0.01379 &  0\\
${13}$& $0$&  $0$&  0 &0   \\
${14}$& $0$&  $0$ & 0.004902&0  \\
${15}$& 0&  $0$& 0.007310 &  0 \\
${16}$& 0& $0$& 0.008850&0  \\
${17}$& $-\frac{1}{4}-\frac{\sqrt{2}}{2}+3\sqrt{2}\xi+\frac{\sqrt{3}}{6}$& $ \frac{2\sqrt{2}-\sqrt{3}-1}{4} $& 0.04284&$ \frac{4-\pi(-\frac{1}{4}-\sqrt{2}+6\sqrt{2}\xi+\frac{\sqrt{3}}{2})}{24}$ \\
${18}$& 0&  0&0.00328 & 0  \\
${19}$& 0&  0& 0.04898 & 0  \\
${20}$& $\frac{1}{2}-6\sqrt{2}\xi+\frac{\sqrt{2}}{2}$&  $\frac{\sqrt{3}-\sqrt{2}}{2}$& 0.07429 &$ -\frac{4+\pi(\frac{1}{2}-12\sqrt{2}\xi+\sqrt{2})}{24}$  \\
${21}$& $-\frac{3}{4}+3\sqrt{2}\xi$&  $\frac{3-\sqrt{3}}{4}$& 0.5730 & $ - \frac{4+\pi(-\frac{3}{4}+6\sqrt{2}\xi)}{24}$ \\
\hline 
\end{tabular}
\caption{Values of $Q_i^j$.}
\label{Qer}
\end{table}

\begin{table}[H]\footnotesize
\centering
\begin{tabular}{cccc}
\hline
$\eta_j$ & $\tilde{w}_j^{(2)}$&$\tilde{w}_j^{(1)}$&$\tilde{w}_j^{(0)}$\\
\hline
$\eta_1$&0&0&0\\
$\eta_2$&0.1777&0.4789& 0.1535\\
$\eta_3$&0&0&0\\
$\eta_4$&0&0&0\\
$\eta_5$&0&0&0\\
$\eta_6$&0&0&0\\
$\eta_7$&0&0&0\\
$\eta_8$&0&0&0\\
$\eta_9$&2.2019&-0.3769&-0.3024\\
$\eta_{10}$&0&0&0\\
$\eta_{11}$&4.7430&1.0450&-0.3830\\
$\eta_{12}$&0&0&0\\
$\eta_{13}$&0&0&0\\
$\eta_{14}$&0&0&0\\
$\eta_{15}$&0&0&0\\
$\eta_{16}$&0&0&0\\
$\eta_{17}$&0.5241&0.0111&-1.937\\
$\eta_{18}$&0&0&0\\
$\eta_{19}$&0&0&0\\
$\eta_{20}$&-1.4678&0.5583&0.2587\\
$\eta_{21}$&1.1620&-0.7321&0.0031\\
$\eta_{22}$&0&0&0\\
\hline
\end{tabular}
\caption{Optimal weights.}
\label{optimal}
\end{table}

\clearpage

\section*{Acknowledgements}
The first author has been supported by the German Research Foundation (DFG) via the Research Group FOR 1548 ``Geometry and Physics of Spatial Random Systems''.
The second author was funded by a grant from the Carlsberg Foundation. The authors are most thankful to Markus Kiderlen who contributed many ideas to this paper.

%
%
%

\end{document}